\documentclass{amsart}

\usepackage[colorlinks=true, urlcolor=black, citecolor=black, linkcolor=black, hyperfootnotes=true]{hyperref}

\usepackage[latin1]{inputenc}
\usepackage{amsfonts}
\usepackage{amsmath}
\usepackage{amsthm}
\usepackage{amssymb}
\usepackage{latexsym}
\usepackage{enumerate}

\newtheorem{theorem}{Theorem}[section]
\newtheorem{lemma}[theorem]{Lemma}
\newtheorem{proposition}[theorem]{Proposition}
\newtheorem{corollary}[theorem]{Corollary}

\newtheorem{question}[theorem]{Question}
\newtheorem{example}[theorem]{Example}
\newtheorem{definition}[theorem]{Definition}

\numberwithin{equation}{section}
\newcounter{ecount}

\begin{document}

\newcommand{\cc}{\mathfrak{c}}
\newcommand{\N}{\mathbb{N}}
\newcommand{\BB}{\mathbb{B}}
\newcommand{\C}{\mathbb{C}}
\newcommand{\Q}{\mathbb{Q}}
\newcommand{\R}{\mathbb{R}}
\newcommand{\T}{\mathbb{T}}
\newcommand{\st}{*}
\newcommand{\PP}{\mathsf{P}}
\newcommand{\lin}{\left\langle}
\newcommand{\rin}{\right\rangle}
\newcommand{\SSS}{\mathbb{S}}
\newcommand{\forces}{\Vdash}
\newcommand{\dom}{\text{dom}}
\newcommand{\osc}{\text{osc}}
\newcommand{\F}{\mathcal{F}}
\newcommand{\A}{A}
\newcommand{\I}{\mathcal{I}}
\newcommand{\B}{\mathcal{B}}
\newcommand{\CC}{\mathcal{C}}

\author{Tristan Bice}
\address{Institute of Mathematics, Polish Academy of Sciences,
ul. \'Sniadeckich 8,  00-656 Warszawa, Poland}
\email{\texttt{Tristan.Bice@gmail.com}}

\author{Piotr Koszmider}
\address{Institute of Mathematics, Polish Academy of Sciences,
ul. \'Sniadeckich 8,  00-656 Warszawa, Poland}
\email{\texttt{piotr.koszmider@impan.pl}}

\subjclass[2010]{}
\title[C*-Algebras With and Without $\ll$-Increasing Approximate Units]{C*-Algebras With and Without\\ $\ll$-Increasing Approximate Units}

%03E35 Consistency and independence results
%46L05 General theory of C*-algebras
%46L85 Noncommutative topology
%54G12 Scattered spaces

\begin{abstract} 
For elements $a, b$ of  a C*-algebra we denote $a=ab$ by $a\ll b$.
We show that all $\omega_1$-unital C*-algebras have $\ll$-increasing 
approximate units, extending a classical result for $\sigma$-unital C*-algebras.  
We also construct (in ZFC) the first examples of C*-algebras with no $\ll$-increasing approximate unit. 
One of these examples is a C*-subalgebra of $\B(\ell_2(\omega_1))$.

These examples are, by necessity, not approximately finite dimensional (AF), 
but some of them are still scattered and so locally finite dimensional (LF) in the sense of Farah and Katsura.  
It follows that there are scattered C*-algebras which are not AF. 

    We further show that the existence of a C*-subalgebra of $\B(\ell_2)$ with no $\ll$-increasing approximate unit
or the existence of an LF but not AF C*-subalgebra of $\B(\ell_2)$  are
     independent of ZFC.
Our examples also show that an extension of an AF algebra by an AF algebra
does not have to be an AF algebra in the nonseparable case.
\end{abstract}

\maketitle

\section{Introduction}

This paper explores certain types of behaviour of nonseparable C*-algebras which have started to be discovered 
only recently and which are not decidable using standard mathematical methods
but rather depend on additional infinitary combinatorial principles. 
A counterexample to Naimark's problem (\cite{aw-naimark}),  pure states which are not multiplicative on any masa
(\cite{aw-masas})
and the existence (or lack thereof) of outer automorphisms in the Calkin algebra (\cite{pw-calkin, ilijas-calkin}) are
some of the very recent milestones of this line of research.

Specifically, we look at certain kinds of approximate units, which have long been a fundamental tool in C*-algebra theory.  First, let $A$ be a C*-algebra and $(\Lambda, \prec)$ a directed set. Recall that a net
$(u_\lambda)_{\lambda\in \Lambda}\subseteq A^1_+$ in the positive unit ball is an \emph{approximate unit} if, for all $a\in A$,
$\|a-au_\lambda\|\rightarrow0$ for $\lambda\in \Lambda$.
As in \cite[II.3.4.3]{Blackadar2017}, we also define a relation $\ll$ on $A$ by
\[a\ll b\qquad\hbox{if and only if}\qquad a=ab.\]
\begin{definition}\label{def-unital}
We call a net $(u_\lambda)_{\lambda\in \Lambda}$ \emph{$\ll$-increasing} if 
$u_\lambda\ll u_\mu$ for $\lambda\prec\mu$.
We call a C*-algebra $A$
\begin{enumerate}
\item \emph{$\ll$-unital} if $A$ contains a $\ll$-increasing approximate unit,\footnote{In Blackadar's textbook
\cite[Definition II.4.1.1]{Blackadar2017}, approximate units that are $\ll$-increasing are called \emph{almost idempotent}, 
but since they have nothing to do with idempotents in general, we adopt
our terminology as more descriptive.}
\item\label{sigmaunital} \emph{$\sigma$-unital} if $A$ has a countable approximate unit,
\item\label{kappaunital} \emph{$\kappa$-unital} if $A$ contains an approximate unit of size $\leq\kappa$, where $\kappa$ is
an infinite cardinal.
\end{enumerate}
\end{definition}
In fact, in most of the paper we work with equivalent notions for subsets
 rather than nets \textendash\, see Proposition \ref{dcof=au}, 
Proposition \ref{cf+dir=inc} and Definition \ref{def-waybelow-unit}.

\begin{proposition}\label{sigma-unital}\cite[Corollary II.4.2.5.]{Blackadar2017}
If a C*-algebra is $\sigma$-unital, then it is $\ll$-unital.
\end{proposition}

Thus it is natural to ask if all C*-algebras are $\ll$-unital. 
 The first main result of this paper strengthens the above proposition to the first uncountable cardinal $\omega_1$:

\begin{theorem}\label{main-omega1}
If a C*-algebra is $\omega_1$-unital, then it is $\ll$-unital.
\end{theorem}

This follows from a strengthening of the classical result that any $\sigma$-unital $A$ is $\ll$-unital, which may be of independent interest even for separable C*-algebras.  
Specifically, we show that any countable $\ll$-increasing approximate 
unit of a C*-subalgebra $B$ of a $\sigma$-unital $A$ extends to such an approximate unit for $A$. 
 This, in turn, relies on extending a classical result on the unitary equivalence 
of close projections \textendash\, see \cite[II.3.3.5]{Blackadar2017}.  
Specifically, we compare the non-symmetric distances $\mathbf{d}$ 
and $\mathbf{u}$ where $\mathbf{d}(a,b)=\|a-ab\|$ measures the 
degree to which $\ll$ holds and $\mathbf{u}$ measures how far
 away from $1$ we must choose a unitary in $\widetilde{A}$ to have $u^*au\ll b$:
\[\mathbf{u}(a,b)=\inf\{\|1-u\|:u\in\widetilde{A},u^*u=1=uu^*,u^*au\ll b\}.\]
It follows from \cite[II.3.3.5]{Blackadar2017} that, for every $\varepsilon>0$, 
there is a $\delta>0$ such that $\mathbf{d}(p,q)<\delta$ 
implies $\mathbf{u}(p,q)<\varepsilon$, for all projections $p$ and $q$. 
 What we require is an extension of this to positive elements, as given in \autoref{<<d<<}.

We also obtain a series of negative results  which refer to C*-algebras of larger densities.
Recall that a  $Q$-set is a subset
of the reals all of whose subsets are relative $G_\delta$-sets.

\begin{theorem}\label{main-negative}\
\begin{enumerate}
\item\label{omega_1} $\B(\ell_2(\omega_1))$ contains a C*-subalgebra of
density $\omega_2$ which is not 
$\ll$-unital, 

\item\label{omega} $\B(\ell_2)$ contains a C*-subalgebra of density $\omega_2$ which is not
$\ll$-unital, if there is an uncountable $Q$-set.
\end{enumerate}
\end{theorem}

\begin{proof}
See Examples \ref{ad-algebra} and \ref{Bell2}.
\end{proof}

Note that \eqref{omega_1} 
does not require any additional set-theoretic assumptions. 
In \eqref{omega} we know that some extra set-theoretic hypothesis is necessary:

\begin{corollary}\label{independent-l2}
Whether $\B(\ell_2)$ contains C*-subalgebras which are not $\ll$-unital
  is independent of ZFC\footnote{ZFC abbreviates the Zermelo-Fraenkel-Choice  axiomatization of 
set theory providing the standard foundations of mathematics. To say that a sentence is independent
of ZFC means that it cannot be proved nor disproved based on the axioms of ZFC. Also ``a ZFC 
example" will mean an example that can be constructed based on the axioms of ZFC.}.
\end{corollary}

\begin{proof} If $2^\omega=\omega_1$ (i.e. CH holds) then all 
C*-subalgebras of $\B(\ell_2)$ are $\ll$-unital, by Theorem \ref{main-omega1}.  
On the other hand, if there is  an uncountable 
$Q$-set then there are non-$\ll$-unital C*-subalgebras of $\B(\ell_2)$, by Theorem \ref{main-negative} \eqref{omega}.
\end{proof}

The main idea of our constructions is
to exploit the dimension drop phenomenon which occurs already in certain C*-subalgebras $D$ of $\ell_\infty(M_2)$,
 as described in the second part of Section 2.2.   In particular, we can have two projections $p,q\in D$
such that there is no $r\in D^1_+$ satisfying $p,q\ll r$ (Lemma \ref{no-way-above}).
The algebra $A_\F$ in Theorem \ref{main-negative} \eqref{omega_1}  can be naturally
 described as a C*-subalgebra of $\ell_\infty^X(D)$ (Those elements
 of the product $D^X$ which have bounded norms), where $X$  is a set and
$\F$ appropriate family of subsets of $X$. 
The structure of the family allows us to apply Lemma \ref{no-way-above} on one of
 the coordinates of any hypothetical $\ll$-increasing approximate unit.  
 Moreover, the example of a  C*-subalgebra of $\B(\ell_2)$  with no $\ll$-increasing  approximate unit
 from Theorem \ref{main-negative} \eqref{omega} has the property that
 it is a C*-subalgebra of $\ell_\infty(D)$ which has $A_\F$ as its quotient.
 Here we exploit a property of some family of subsets of $\N$
   obtained from a $Q$-set, using it to `code' the algebra $A_\F$ inside the algebra $\ell_\infty(D)$.
 In particular, all these
examples are $2$-subhomogenous.

Our results are also related to finiteness notions for C*-algebras. Concerning 
AF and LF algebras we follow the terminology of \cite{farah-katsura}
and not the one  of \cite{kusuda} and \cite{lin}:

\begin{definition}\label{af-lf}
A C*-algebra  $A$ is called 
\begin{itemize}
\item  \emph{approximately finite-dimensional (AF)}  if it has a directed family of finite-dimensional
C*-subalgebras with dense union;
\item \emph{locally finite-dimensional (LF)} if, for any finite 
subset  $F$ of $A$ and any $\varepsilon>0$, there is
a finite-dimensional C*-subalgebra $B$ of $A$ with $F\subseteq_\varepsilon B$.
\item \emph{almost unital (AU)} if it has a directed family of unital
C*-subalgebras with dense union;
\item \emph{locally unital (LU)} if, for any finite subset $F$ of $A$ and any $\varepsilon>0$, there is
a unital C*-subalgebra $B$ of $A$ with $F\subseteq_\varepsilon B$.
\end{itemize}
\end{definition}
Here $F\subseteq_\varepsilon B$ means that for every $f\in F$ there is $b\in B$ such that $\|f-b\|<\varepsilon$.  
By `unital C*-subalgebra' we mean a C*-algebra $B\subseteq A$ with its own unit $1_B$, not 
necessarily a unit for $A$.  The only thing we know about each $1_B$ in the larger algebra is that it is a projection. 
 Indeed, each projection $p\in\mathcal{P}(A)$ is a unit for the (hereditary) C*-subalgebra $pAp$.

In \cite[Theorem 2.2]{bratteli}, Bratteli showed that AF and LF are equivalent for separable C*-algebras.  
In \cite[Theorem 1.5]{farah-katsura}, Farah and Katsura showed that AF and LF are equivalent for 
C*-algebras of density $\omega_1$ but not for all larger densities.

As finite-dimensional C*-algebras are unital, every AF-algebra is AU, and certainly every AU-algebra is $\ll$-unital. 
 So our examples of C*-algebras which are not $\ll$-unital can not be AF, but some of them are LF, 
since they are scattered and scattered algebras are LF by a result of Lin (see \cite[Lemma 5.1]
{lin}, where LF is called AF).  Thus our examples differ from the non-AF LF-algebras 
constructed in \cite[\S6]{farah-katsura} which can not be scattered because they 
contain a copy of the CAR algebra (which has no minimal projections).

\begin{corollary}\label{Bell2notLF}
Whether $\B(\ell_2)$ contains C*-subalgebras which are LF but not AF or AU is independent of ZFC.
\end{corollary}

\begin{proof} If $2^\omega=\omega_1$ then all C*-subalgebras of $\B(\ell_2)$ which are LF are AF
as well, by \cite[Theorem 1.5]{farah-katsura}.  On the other hand, 
by Example \ref{Bell2scattered} and the above-mentioned result of Lin we can consistently have LF
but not AF C*-subalgebras of $\B(\ell_2)$.
\end{proof}

Note that a related question of Takesaki if it is consistent that   there are subalgebras of $\B(\ell_2)$ 
which are locally matricial (i.e. matroid, in Dixmier's terminology) but not approximately matricial (7.19 of \cite{farah-katsura}) remains
unresolved.

It turns out that the properties of being almost unital and locally unital are related 
to our investigation of $\ll$-unital algebras, namely we obtain:

\begin{theorem}\label{intro-aulu} Let $A$ be a C*-algebra.
\begin{enumerate}
\item $A$ is LU if and only if $A$ has an approximate unit consisting of projections.
\item $A$ is AU if and only if $A$ has a $\ll$-increasing  approximate unit consisting of projections.
\item $A$ is AU if and only if $A$ is LU and has a $\ll$-increasing approximate unit.
\end{enumerate}
\end{theorem}
\begin{proof} Apply Propositions \ref{LUequivs} and \ref{au=ll-unit-of-projections} and Theorem \ref{au=lu+waybelow}.
\end{proof}

Our results also relate to scattered C*-algebras, which generalize scattered 
locally compact spaces (where every non-empty subset contains an isolated point).
There are many equivalent conditions for being scattered, which are surveyed in \cite{cantor-bendixson}.  
For example, $A$ is scattered if and only if every nonzero quotient of $A$ contains a nonzero minimal projection, 
i.e. a projection $p$ such that $pAp=\C p$.
Kusuda proved that a C*-algebra $A$ is scattered if and only if every C*-subalgebra
of $A$ is LF (which Kusuda calls AF).  A natural question
is whether AF and LF can be distinguished in such a strong finiteness environment.
Since some of our ZFC examples are scattered, namely Examples \ref{crosses-algebra} and \ref{ex-tree-zfc},  we obtain the following positive answer.

\begin{theorem}\label{scattered-notAU} There are scattered C*-algebras which are not AU, in particular
they are not AF.  
\end{theorem}

Restrictions on the densities of such algebras and the densities of the Hilbert spaces where these
algebras can be represented are detailed in the following theorem. Recall that a Canadian tree 
(also known as a weak Kurepa tree) is a tree of height $\omega_1$, with each level
of cardinality at most $\omega_1$ and with more than $\omega_1$
uncountable branches (see Section 2.3 for more details). 
It is known that the existence of Kurepa or Canadian trees is independent 
of the continuum hypothesis (CH) and Martin's axiom (MA) 
(see \cite{devlin}, \cite{stevo}).  
Uncountable $Q$-sets are also
known to be consistent with Kurepa and hence Canadian trees
(see the discussion after Definition \ref{separated}). 

\begin{theorem}\
\begin{enumerate}
\item There are scattered C*-subalgebras of $\B(\ell_2(\omega_2))$ of density $\omega_2$ which are not AF,
\item There are scattered C*-subalgebras of $\B(\ell_2(2^\omega))$ of density bigger than $2^\omega$ which are not AF,
\item There are scattered C*-subalgebras of $\B(\ell_2(\omega_1))$ of density bigger than $\omega_1$ which are not AF,
if there exists a Canadian tree,
\item There are scattered C*-subalgebras of $\B(\ell_2)$ of density bigger than $\omega_1$ which are not AF,
if there exists a Canadian tree and a Boolean embedding $E: \wp(\omega_1)\rightarrow\wp(\N)/Fin$,
\end{enumerate}
\end{theorem}
\begin{proof} See Examples \ref{crosses-algebra}, \ref{ex-tree-zfc}, \ref{ex-tree-can} and \ref{Bell2scattered}.
\end{proof}

Taking into account the result of \cite{farah-katsura} we obtain the following:

\begin{theorem}
Whether there are C*-subalgebras of $\B(\ell_2)$ that are scattered but not AF is independent of ZFC.
\end{theorem}

However, we can not say the same for subalgebras of $\B(\ell_2(\omega_1))$, which raises the following question.

\begin{question}
Do there exist in ZFC   scattered C*-subalgebras of $\B(\ell_2(\omega_1))$ which are not AF?
\end{question}

A negative answer might be obtained from additional assumptions inconsistent with Canadian trees, 
e.g. the Proper Forcing Axiom (PFA) \textendash\, see \cite[7.10]{baumgartner}. Also
the Mitchell model seems a natural place where there may exist 
scattered C*-subalgebras of $\B(\ell_2(\omega_1))$ which are not AF, because 
it is already known that some consequences of Canadian trees
concerning commutative scattered C*-algebras are false there (see \cite{bs}, \cite{weese}).

Using the same arguments as for separable C*-algebras one can show that quotients and ideals
of AF algebras are again AF in the nonseparable context (cf. \cite[III.4.1, 4.4]{davidson}).
As our C*-algebras are extensions of AF algebras by AF algebras we obtain the following:

\begin{theorem} There are extensions of AF algebras by AF algebras which are not AF,
while extensions of LF algebras by LF algebras are always LF.
\end{theorem}
\begin{proof} 
Apply Proposition \ref{extension} to either of our scattered ZFC Examples \ref{crosses-algebra} or 
\ref{ex-tree-zfc}.

To prove the second part of the theorem suppose that $A$ is a C*-algebra, $I\subseteq A$ is its ideal and both $I$ and $A/I$
are LF.  Let $F$ be a finite subset of $A$ and $\varepsilon>0$. It is enough to find a
separable algebra $B\subseteq A$ containing $a$ such that $B\cap I$ and $B/B\cap I$
are LF.  Then by Bratteli's result $B\cap I$ and $B/B\cap I$ 
are AF and so $B$ is AF since being AF  is preserved by extensions of
separable C*-algebras (\cite[II.6.3]{davidson}), hence there is a finite-dimensional $C\subseteq B$
such that $F\subseteq_\varepsilon C$ which provides the required approximation
of $F$ in $A$.

To construct $B$, for each $n\in\N$, recursively build separable C*-subalgebras $B_n\subseteq B_{n+1}$ 
of $A$ with countable  sets $D_n$ and $E_n$ dense in $B_n\cap I$
and $ B_n/B_n\cap I$ respectively with
$F\subseteq B_0$ such that, for each finite $G\subseteq D_n$ and $k\in \N$, there is
a finite dimensional C*-subalgebra $B_{G, k}\subseteq B_{n+1}\cap I$ with $G\subseteq_{1/k} B_{G, k}$
and, for each finite $H\subseteq E_n$ and $k\in \N$, there is
a finite dimensional C*-subalgebra $C_{H, k}\subseteq B_{n+1}/B_{n+1}\cap I$ with $H\subseteq_{1/k} C_{H, k}$.
Then put $B=\overline{\bigcup_{n\in \N}B_n}$.
\end{proof}

The above theorem contrasts with the separable case where AF is preserved by extensions (\cite[II.6.3]{davidson}).
The nonseparable case seems quite different in this context, for example it was proved
in \cite{mrowka} and \cite{stable} that stability of AF algebras is not
preserved by extensions nor by simple uncountable inductive limits.  However, we do not know
the answer to the following:

\begin{question} Is being AF preserved by
(linear) inductive limits of length $\omega_1$ (or even $\omega$)?
\end{question}

The structure of the paper is as follows. The next section consists of preliminaries.
The third section aims at proving
results which yield Theorem \ref{intro-aulu}. The fourth
section concerns $\omega_1$-unital algebras and
 the proof of  Theorem \ref{main-omega1}.
The last section contains the counterexamples required for Theorem \ref{main-negative}, Corollary \ref{independent-l2}
and Theorem \ref{scattered-notAU}.

For general background, see \cite{Blackadar2017} or \cite{pedersen} for C*-algebras and \cite{jech} or \cite{kunen} for set theory.  Throughout, we consider a C*-algebra $A$ and use the following notation:
$A^1$ for the unit ball of $A$, $A_+$ for the set of positive elements of $A$, $A^1_+$ for
the positive part of the unit ball, $A^{=1}$ for the unit sphere of $A$, $\mathcal P(A)$ 
for the set of projections in $A$ and $\widetilde A$ for the unitization of $A$ (if $A$ is not already unital, otherwise we take $\widetilde A=A$).
For any $a\in A$, we let $a^\perp=1-a\in\widetilde A$ and we denote the spectrum of $a$ by $\sigma(a)$.  We say that
a Banach space $B$ has density $\kappa$ if $\kappa$ is the minimum cardinality of a norm dense subset of $B$.
We let $\ell_2(\kappa)$ denote the Hilbert space of square summable functions on $\kappa$, which is the unique Hilbert space of density $\kappa$.  We also let $\B(H)$
denote the C*-algebra of all bounded linear operators on a Hilbert space $H$.
If $A$ is a C*-algebra and $X$ is a set, by $\ell_\infty^X(A)$ we mean the
C*-algebra of all norm bounded functions from $X$ into $A$ with pointwise operations.
If $X=\N$, we write $\ell_\infty(A)$ for $\ell_\infty^\N(A)$.

We let $\omega_1$ denote the first uncountable cardinal, and
$\omega_2$ the  second uncountable cardinal while $2^\omega$ denotes the cardinality of the continuum.
$[X]^{\leq\omega}$ denotes the family of all  finite and infinite countable  subsets of $X$.
The continuum hypothesis CH is the assertion that $2^\omega=\omega_1$.
Given a transitive relation $<$ on a set $P$, we say that $Q\subseteq P$ is directed 
if, for every $p, q\in Q$, there is some $r\in Q$ such that $p,q<r$; $Q$ is cofinal 
if, for every $p\in P$, there is some $q\in Q$ such that $p<q$ and $Q$ is coinitial if,
for every $p\in P$, there is some $q\in Q$ such that $q<p$.
When we consider several partial orders, then $<$-directed, $<$-cofinal, $<$-coinitial are
the corresponding notions for $<$.

We would like to thank I. Farah and J. Steprans for valuable comments which improved the paper.

\section{Preliminaries}

\subsection{Approximate relations and approximate units}

\begin{definition} Suppose that $A$ is a C*-algebra, $a, b\in A$,   $U\subseteq A^1_+$
and $\varepsilon>0$.
\begin{enumerate}
\item $a\ll_\varepsilon b$ if and only if $\|a-ab\|<\varepsilon$,
\item $U$ is $\ll$-approximately cofinal in $B\subseteq A$ if for every $b\in B$ and every $\varepsilon>0$
there is $u\in U$ with $b\ll_\varepsilon u$,
\item $U$ is $\ll$-approximately directed   if for every $a, b\in U$ and every $\varepsilon>0$
there is $u\in U$ with $a, b\ll_\varepsilon u$.
\end{enumerate}
\end{definition}

\begin{proposition}\label{triangle} Suppose that $A$ is a C*-algebra.
For all $a,b,c\in A^1_+$,
$$\|a-ab\|\leq \|a-ac\|+\|c-cb\|,$$
and consequently 
for all $\varepsilon,\delta>0$ we have that
$a\ll_\varepsilon c\ll_\delta b$ implies $a\ll_{\varepsilon+\delta}b$.
\end{proposition}

\begin{proof}
As $\|a\|,\|b^\perp\|\leq1$,
\[\|a-ab\|=\|ab^\perp\|=
\|a(c^\perp+c)b^\perp\|\leq\|ac^\perp\|||b^\perp\|+\|a\|||cb^\perp\|\leq\|a-ac\|+\|c-cb\|.\qedhere\]
\end{proof}

\begin{proposition}\label{cofinal-directed} Let $A$ be a $C^*$-algebra.
$\ll$-approximately cofinal sets in $A^1_+$  are $\ll$-approximately directed.
\end{proposition}
\begin{proof}
Suppose that a subset $U$ of a C*-algebra is $\ll$-approximately cofinal in $A^1_+$.
 To see that $U$ is $\ll$-approximately directed, take $a,b\in U$ and let $c=\frac{1}{2}(a+b)$.  For any $u\in A^1_+$,
\[\|a-au\|^2=\|u^\perp a^2 u^\perp\|\leq\|u^\perp au^\perp\|
\leq2\|u^\perp cu^\perp\|\leq2\|cu^\perp\|=2\|c-cu\|.\]
Likewise, $\|b-bu\|^2\leq2\|c-cu\|$.  Now let $\varepsilon>0$.
By $\ll$-approximate cofinality of $U$ in $A^1_+$ there is $u\in U$ such that
$\|c-cu\|<\varepsilon^2/2$ and so $\|a-au\|, \|b-bu\|<\varepsilon$
which proves that $U$ is $\ll$-approximately directed.
\end{proof}

By a \emph{net} we mean a set $(u_\lambda)_{\lambda\in\Lambda}$ indexed by a directed set $(\Lambda,\prec)$, for some transitive relation $\prec$.  Some authors prefer the relation $\prec$ to also be reflexive but this is unnecessary.  In any case, if so desired, one can always consider the relation $\preceq$ instead where
$\lambda\preceq\mu$ if and only if $\lambda\prec\mu$ or $\lambda=\mu$.
In general, the indexing set $\Lambda$ of a net $(u_\lambda)_{\lambda\in\Lambda}$ is of vital importance. 
 For example, if we have a net $(r_\lambda)_{\lambda\in \Lambda}\subseteq\mathbb{R}$
  and all we know is that the underlying set $\{r_\lambda:\lambda\in\Lambda\}$ 
  is the entirety of $\mathbb{Q}$ then we have no idea what $(r_\lambda)_{\lambda\in \Lambda}$ 
  converges to or even if it converges at all.  However, with approximate units, 
  this is not the case \textendash\, they always converge to $1\in A^{**}$ 
  (w.r.t. the weak, weak* or strong topology) and the indexing set is often irrelevant. 
   This basic but important point is often overlooked, and one we wish to make precise  in this section.  
   This will also allow us to work with subsets 
   rather than nets and avoid irrelevant questions related 
   to the indexing set (e.g. if $(u_\lambda)$ is an approximate 
   unit for $A$ with 
   $|\{u_\lambda:\lambda\in\Lambda\}|=\omega_1$ but $|\Lambda|=\omega_2$,
    then is $A$ really $\omega_1$-unital?  The answer is yes, as we shall soon see).

\begin{proposition}\label{dcof=au}
Let $A$ be a $C^*$-algebra and   $U\subseteq A^1_+$. 
$U$ is $\ll$-approximately cofinal in $A^1_+$ if and only if there is a directed set $\Lambda$ 
and  an approximate unit $(u_\lambda)_{\lambda\in \Lambda}$ such that
$U=\{u_\lambda:\lambda\in\Lambda\}$.  Moreover, $\Lambda$ can be taken
to have cardinality equal to the cardinality of $U$.
\end{proposition}

\begin{proof}
If $(u_\lambda)_{\lambda\in\Lambda}$ is an approximate unit, then
 $\{u_\lambda:\lambda\in\Lambda\}$ is immediately seen
  to be $\ll$-approximately cofinal.  Conversely, say $U$ is $\ll$-approximately
  cofinal in $A^1_+$ and consider $\Lambda=U\times\mathbb{N}$ ordered by
\[(a,m)\prec(b,n)\qquad\hbox{if and only if}\qquad\|a-ab\|<\tfrac{1}{m}-\tfrac{1}{n}.\]
We first claim that $\prec$ is transitive.  Indeed, if $(a,m)\prec(c,o)\prec(b,n)$ then, by \ref{triangle},
\[\|a-ab\|\leq\|a-ac\|+\|c-cb\|<\tfrac{1}{m}-\tfrac{1}{o}+\tfrac{1}{o}-\tfrac{1}{n}=\tfrac{1}{m}-\tfrac{1}{n},\]
i.e. $(a,m)\prec(b,n)$.

We next claim that $\Lambda$ is $\prec$-directed.  As $U$ is $\ll$-approximately directed
by Proposition \ref{cofinal-directed}, for any $(a,m),(b,n)\in\Lambda$, we can take $o>m+n$ and $u\in U$ with
\[\|a-au\|+\|b-bu\|<\tfrac{1}{m+n}-\tfrac{1}{o},\]
so $(a,m),(b,n)\prec(u,o)$, by the definition of $\prec$, i.e. $\Lambda$ is $\prec$-directed.

For each $\lambda=(a,n)\in\Lambda$, set $u_\lambda=a$.  
We claim that the resulting net $(u_\lambda)_{\lambda\in \Lambda}$ is an approximate unit.  
To see this, take any $a\in A^1_+$ and $n\in\mathbb{N}$.  
As $U$ is $\ll$-approximately cofinal in $A^1_+$, we have $u\in U$ with $\|a-au\|<\frac{1}{n}$.  
So $(u,n)\in\Lambda$ and by \ref{triangle}, for any $(w,m)\succ(u,n)$,
\[\|a-aw\|\leq\|a-au\|+\|u-uw\|<\tfrac{1}{n}+\tfrac{1}{n}-\tfrac{1}{m}<\tfrac{2}{n}.\]
Thus $\|a-au_\lambda\|$ converges to $0$ for $\lambda\in \Lambda$.  For any $a\in A^1$, we have $a^*a\in A^1_+$ so 
\[\|a-au_\lambda\|^2=
\|u_\lambda^\perp a^*au_\lambda^\perp\|\leq\|a^*au_\lambda^\perp\|=
\|a^*a-a^*au_\lambda\|\rightarrow0.\]
Finally, for any non-zero $a\in A$, 
we have $\|a-au_\lambda\|$ converges to zero for $\lambda\in \Lambda$
by the homogeneity of the norm,
so $(u_\lambda)_{\lambda\in \Lambda}$ is indeed an approximate unit.

When $U$ is infinite, we have $|\Lambda|=|U\times\mathbb{N}|=|U|$.  
When $U$ is finite, $\ll$-approximate cofinality means that we have a single element
 $u\in U$ with $\|a-au\|=0$, i.e. $a\ll u$, for all $a\in A^1_+$ (and hence $u=1$).  Thus, taking $\Lambda=U$ and $\prec$
equal to $\ll$, we again have an approximate unit $(u_u)_{u\in\Lambda}$ with $|\Lambda|=|U|$.
\end{proof}

In  Proposition \ref{cofinal-directed}, 
we saw that every $\ll$-approximately cofinal $U$ is $\ll$-approximately directed.  
We are interested in $U$ satisfying the stronger notion of being $\ll$-directed.

\begin{proposition}\label{cf+dir=inc} Suppose that 
$A$ is a C*-algebra and let $U\subseteq A^1_+$. The following are equivalent:
\begin{enumerate}
\item $U$ is $\ll$-approximately cofinal in $A$ and $\ll$-directed,
\item  There is a $\ll$-increasing approximate unit $(u_\lambda)_{\lambda\in \Lambda}$
with $U=\{u_\lambda:\lambda\in\Lambda\}$.
\end{enumerate}
\end{proposition}
\begin{proof}
For the forward implication take $\Lambda=U$ and 
order it by $\ll$.  First take $a\in A^1_+$ and $\varepsilon>0$ and find $u\in U$ satisfying $\|a-au\|<\varepsilon$.
If $u\ll w$, by Proposition \ref{triangle} we have
$$\|a-aw\|\leq \|a-au\|+\|u-uw\|=\|a-au\|<\varepsilon.$$
The argument for remaining $a\in A$ that $\|a-au\|$ converges to zero for $u\in U$ is as at the end of
the proof of Proposition \ref{dcof=au}.
For the converse implication note that $U$ is $\ll$-directed
as $(\Lambda, \prec)$ is directed and $\lambda\prec\mu$ implies $u_\lambda\ll u_\mu$.
It is also clear that $U$ has to be $\ll$-approximately cofinal.
\end{proof}
Thus we are justified to introduce the following:
\begin{definition}\label{def-waybelow-unit} Suppose that $A$ is a C*-algebra.
$U\subseteq A^1_+$ is called a $\ll$-unit for $A$ if
it is $\ll$-approximately cofinal in $\A$ and $\ll$-directed.
\end{definition}

It follows that $A$ contains a $\ll$-unit if and only if 
$A$ is $\ll$-unital in the sense of Definition \ref{def-unital}.

The following elementary observation will be needed later in section 4.

\begin{lemma}\label{aiau-quotient}
If a C*-algebra $A$ is $\ll$-unital then so are all its quotients.
\end{lemma}

\begin{proof}
If $B$ is a quotient of $A$ and $\pi$ is the canonical homomorphism from $A$ onto $B$ then 
$\|\pi(a)-\pi(a)\pi(b)\|\leq\|a-ab\|$, as $\pi$ is norm decreasing.  In particular, 
$\ll$-approximate cofinality is preseverved by $\pi$, as is $\ll$.  Thus $\pi$ takes any $\ll$-unit in $A$ to a $\ll$-unit in $B$.
\end{proof}

\noindent{\bf Remarks:} The key in this subsection was to analyze the binary function $\mathbf{d}$ on $A$ defined by
\[\mathbf{d}(a,b)=\|a-ab\|.\]
The first thing to note is that, as with metrics, $\mathbf{d}$ satisfies a triangle inequality on $A^1_+$,
as shown in Proposition \ref{triangle} and \cite[Proposition 1.2]{bice-vignati} (as $\mathbf{d}$ is not symmetric, this is just one of several possible triangle inequalities, but this is the one that corresponds to transitivity).  Note `$\ll$-approximately cofinal' and `$\ll$-approximately directed' are called `$\mathbf{d}$-cofinal' and `$\mathbf{d}$-directed' in \cite{bice-vignati}, where various other `approximate' order properties relative to $\mathbf{d}$ and $\mathbf{h}(a,b)=\|(a-b)_+\|$ are also considered.

It is a classical fact that every C*-algebra possesses an approximate unit, 
and hence a $\ll$-approximately cofinal subset.  In fact, Proposition \ref{dcof=au} provides a somewhat non-standard proof of this, once we note that the entirety of $A^1_+$ is always $\ll$-approximately cofinal.  Indeed, for all $a\in A^1_+$,
\[\inf_{u\in A^1_+}\|a-au\|\leq\lim\|a-a^{1/n}a\|=0.\]

\subsection{The Pedersen Ideal}

As in \cite[II.5.2.4]{Blackadar2017}, we consider $A_\mathrm{c}\subseteq A_+$ given by
\begin{align}
\nonumber A_{\mathrm{c}}&=\{a\in A_+:\exists b\in A_+\ (a\ll b)\}.\\
\nonumber &=\{(a-\varepsilon)_+:a\in A_+\text{ and }\varepsilon>0\}.\\
\label{f} &=\{f(a):a\in A_+\text{ and }f\in C_0((0,\infty))_\mathrm{c}\}.
\end{align}
The ideal generated by $A_\mathrm{c}$ is the \emph{Pedersen ideal} $\mathrm{Ped}(A)$, the smallest dense ideal in $A$.  Any $\ll$-unit must be contained in $A_\mathrm{c}^1=A_\mathrm{c}\cap A^1$, and sometimes $A_\mathrm{c}^1$ itself forms a $\ll$-unit.

\begin{proposition}
$A_c^1$ is a $\ll$-unit if and only if $A_\mathrm{c}=\mathrm{Ped}(A)_+$.
\end{proposition}

\begin{proof}
If $A_\mathrm{c}=\mathrm{Ped}(A)_+$ then, for any $a,b\in A_\mathrm{c}^1$, we have $a+b\in A_\mathrm{c}$ and hence $a+b\ll c$, for some $c\in A_+$.  Taking $f,g\in C_0((0,1])_+$ with $f\ll g$ and $f(1)=1$, we have $a+b\ll f(c)\ll g(c)$ and hence $a,b\ll f(c)\in A_\mathrm{c}^1$, i.e. $A_\mathrm{c}^1$ is $\ll$-directed.  By \eqref{f}, $A_\mathrm{c}^1$ is approximately $\ll$-cofinal and hence a $\ll$-unit.

Conversely, if $A_\mathrm{c}^1$ is a $\ll$-unit then $A_\mathrm{c}$ is an additive cone.  Indeed, for any $a,b\in A_\mathrm{c}^1$, we have $c\in A_\mathrm{c}^1$ with $a,b\ll c$ and hence $a+b\ll c$, so $a+b\in A_\mathrm{c}$.  Thus, by \cite[Theorem 5.6.1]{pedersen},
\begin{align*}
\mathrm{Ped}(A)_+&=\{a\in A_+:a\leq b_1+\cdots+b_k\text{ and }b_1,\cdots,b_k\in A_\mathrm{c}\}\\
&=\{a\in A_+:a\leq b\in A_\mathrm{c}\}\\
&=A_\mathrm{c}.\qedhere
\end{align*}
\end{proof}

When $A=C_0(X)$, for some locally compact $X$, $A_\mathrm{c}$ consists precisely of the positive elements with compact support.  As the compactly supported elements form an ideal (namely the Pedersen ideal) we have the following:

\begin{corollary}
If $A$ is commutative then $A_\mathrm{c}^1$ is a $\ll$-unit.
\end{corollary}

However, $A_\mathrm{c}=\mathrm{Ped}(A)_+$ can also hold for non-commutative non-unital $A$, e.g. $C_0(\mathbb{N},M_2)$.  To obtain $A$ where $A_\mathrm{c}\neq\mathrm{Ped}(A)_+$, we extend $C_0(\mathbb{N},M_2)$ by one dimension to the C*-algebra $D$ of convergent $M_2$-sequences (with pointwise addition etc.) where the dimension drops to one in the limit.  More precisely, for any $\theta\in\mathbb{R}$, consider the projection $P_\theta\in M_2$ onto $\mathbb{C}(\sin\theta,\cos\theta)$, i.e.
\[P_\theta=\begin{bmatrix}\sin^2\theta&\sin\theta\cos\theta\\ \sin\theta\cos\theta&\cos^2\theta\end{bmatrix}\]
and define $D$ by
\begin{equation}\label{Ddef}
D=\{(a_n)_{n\in\N}\in \ell_\infty(M_2):\lim_{n\rightarrow\infty} a_n\in\mathbb{C}P_0\}.
\end{equation}

\begin{lemma}\label{scattered} $D$ is a scattered C*-algebra.
\end{lemma}
\begin{proof}  We may us the fact that a C*-algebra $A$ is scattered if, and only if, 
every selfadjoint element of $A$ has countable spectrum (1.4 of \cite{cantor-bendixson}).
If $(a_n)_{n\in\N}\in D$ is selfadjoint
and $\lim_{n\rightarrow\infty} a_n=zP_0$ for some $z\in \C$, then the elements
of the spectra of the $a_n$s must converge to $\{z, 0\}$ which is the spectrum of $zP_0$.
It follows that the closure $X$ of $\bigcup_{n\in \N}\sigma(a_n)$ is countable. Now if
$\lambda\not \in X$, then $(a_n)_{n\in\N}-\lambda1$ is invertible which completes the proof.
\end{proof}

Now consider the projections $p=(p_n)_{n\in \N}$ and $q=(q_n)_{n\in \N}$ in $D$
 defined by $p_n=P_0$ and $q_n=P_{1/n}$ for each $n\in \N$.

\begin{lemma}\label{no-way-above}
There is no $a\in D^1_+$ with $p, q\ll a$.
\end{lemma}

\begin{proof}  Aiming for a contradiction, assume that $p, q\ll a$, i.e. $p_na_n=p_n$ and $q_na_n=q_n$, for all $n\in\N$.
  In particular, $a_n$ commutes with both $p_n$ and $q_n$, as all the elements are self-adjoint.  This implies that $a_n$ can be simultaneously diagonalized with both $P_0$ and $P_{1/n}$ so $(0,1)$ and $(\sin(1/n),\cos(1/n))$ are eigenvectors for $a_n$.  As $p_na_n=p_n$ and $q_na_n=q_n$, the corresponding eigenvalues must both be $1$.  Thus $a_n=1\in M_2$, for all $n\in\N$, 
  so $\lim_{n\rightarrow\infty}a_n=1$, contradicting the definition of $D$.
\end{proof}

Note that the projections
\[\mathcal{P}(A)=\{p\in A_+:p\ll p\}\]
are always contained in $A_\mathrm{c}$.  Thus $D_\mathrm{c}$ is not a $\ll$-unit, by Lemma \ref{no-way-above}, and hence $D_\mathrm{c}\neq\mathrm{Ped}(D)_+$ (in fact $\mathrm{Ped}(D)=D$ in this case).  Nonetheless, it is not hard to explicitly define a smaller $\ll$-unit for $D$, e.g. the collection of $(u_n)_{n\in \N}\subseteq D$ of the form
\[u_n(m)=\begin{cases}1&\text{if }m<n,\\ P_0&\text{otherwise}.\end{cases}\]

\subsection{Trees}  To describe our examples in
the last section more precisely, we first need some basic tree related terminology and theory (see \cite{stevo-trees} for more background).

A \emph{tree} is a set $T$ together with a strict partial order $\prec$ such that $\{t:t\prec s\}$ is well-ordered, for all $s\in T$, i.e. $\prec$ is transitive, $s\nprec t$ or $t\nprec s$, for all $s,t\in T$, and every $S\subseteq\{t:t\prec s\}$ has a unique $\prec$-minimal element.  In particular, every pair $t,u\prec s$ has a $\prec$-minimum, i.e. $t\prec u$ or $u\prec t$, so $\{t:t\prec s\}$ is always linearly ordered.  In fact, $\{t:t\prec s\}$ is always order isomorphic to a unique ordinal $\alpha$, which we call the \emph{height} of $s$, denoted by $ht(s)$.  The height of $S\subseteq T$ is then defined by
\[ht(S)=\sup_{s\in S}(ht(s)+1).\]
In particular, $ht(s)=ht(\{t:t\prec s\})$.  The $\alpha^\mathrm{th}$ level is denoted by
\[Lev_\alpha(T)=\{s\in T:ht(s)=\alpha\}.\]
In particular, $ht(T)=\min\{\alpha:Lev_\alpha(T)=\emptyset\}$. 

A \emph{branch} of a tree $T$ is a maximal linearly ordered subset and a \emph{$\kappa$-branch} is a branch of height $\kappa$.  The sets of all branches and all $\kappa$-branches are denoted by $Br(T)$ and $Br_\kappa(T)$ respectively.  
Note that the linearity
of the tree order on the sets $\{t:t\prec s\}$ for $s\in T$ means that, for any branches $b,b'\in Br_{ht(T)}(T)$, $b\cap Lev_{\alpha'}(T)\not=b'\cap Lev_{\alpha'}(T)$ implies $b\cap Lev_\alpha(T)\not=b'\cap Lev_\alpha(T)$ whenver $\alpha'\leq\alpha<ht(T)$.
 The \emph{exit set} of a branch $b$ is defined by
\[b^\#=\{t\in T\setminus b:\{s:s\prec t\}\subseteq b\}.\]
We call a tree $T$ a
\begin{enumerate}
\item \emph{$\kappa$-tree} if $ht(T)=\kappa$ and $|Lev_\alpha(T)|<\kappa$, for all $\alpha<\kappa$.
\item \emph{Kurepa tree} if $T$ is a $\omega_1$-tree with $|Br_{\omega_1}(T)|>\omega_1$.
\item \emph{Canadian tree} if $ht(T)=\omega_1\geq Lev_\alpha(T)$, for all $\alpha$, and $|Br_{\omega_1}(T)|>\omega_1$.
\end{enumerate}

We consider $\{0,1\}^{<\kappa}=\{f\in \{0,1\}^\alpha:\alpha<\kappa\}$ as a tree ordered by end-extensions, i.e. $f\prec g$ if $\dom(f)\subseteq\dom(g)$ and $g|\dom(f)=f$.  We identify $\{0,1\}^\kappa$ with $Br(\{0,1\}^{<\kappa})$ via the map $f\rightarrow\{f|\alpha:\alpha<\kappa\}$.

\begin{lemma}\label{minimal-cardinal}
Let $\kappa$ be the minimal cardinal with $2^\omega<2^\kappa$.
\begin{enumerate}
\item $cf(\kappa)>\omega$.
%\item $cf(2^\kappa)>\kappa$.
\item $\kappa\leq 2^\omega$.
\setcounter{ecount}{\value{enumi}}
\end{enumerate}
Moreover, $(\{0,1\}^{<\kappa}, \prec)$ is a tree such that
\begin{enumerate}
\setcounter{enumi}{\value{ecount}}
\item $ht(T)=\kappa$,
\item $T$ has  cardinality $2^\omega$ 
%(in particular, $Lev_\alpha(T)\leq 2^\omega$, for all $\alpha<ht(T)=\kappa$).
\item $T$ has more than $2^\omega$  $\kappa$-branches.
\end{enumerate}
\end{lemma}

\begin{proof}\
\begin{enumerate}
\item Suppose $\kappa=\sup_{n\in \N}\lambda_n$ for some $\lambda_n<\kappa$ and $n\in \N$. It follows that there is an injection of $2^\kappa$ into $\Pi_{n\in \N} 2^{\lambda_n}$ (just take $(X\cap \lambda_n)_{n\in \N}$ as an element associated to $X\subseteq \kappa$).  By the minimality of $\kappa$ we have that $2^{\lambda_n}\leq 2^\omega$.
So we have
\[2^\kappa\leq (2^\omega)^\omega=2^\omega,\]
which contradicts the property that $2^\kappa>2^\omega$.

%\item This is a well known corollary of Konig's lemma, even for arbitrary $\kappa$ (5.10, 5.12 of \cite{jech}).

\item Immediate from $2^{2^\omega}>2^\omega$ and the definition of $\kappa$.
\item is clear.
\item follows from (2), (3), the fact that all levels of $T$ are not bigger than $2^\omega$ and $|Lev_\omega(T)|=2^\omega$.
\item follows from the fact that $2^\kappa>2^\omega$.
\qedhere
\end{enumerate}
\end{proof}

\section{\texorpdfstring{$\ll$-units consisting of projections}{<<-units consisting of projections}}

Projections have long played an important role in operator algebra theory and there are various properties quantifying the amount of projections a C*-algebra possesses (see \cite{Blackadar1994}).  Here we briefly examine two such properties, namely the almost unital (AU) and locally unital (LU) C*-algebras in Definition \ref{af-lf}.  These generalize the approximately finite dimensional (AF) and locally finite dimensional (LF)
C*-algebras in the same definition. First we characterize LU algebras.

\begin{proposition}\label{LUequivs} For a C*-algebra $A$
the following are equivalent:
\begin{enumerate}
\item\label{LU} $A$ is LU.
\item\label{auproj} $A$ has an approximate unit consisting of projections.
\item\label{dense} $A=\overline{\bigcup\mathcal{F}}$ for the family $\mathcal{F}$ of all unital C*-subalgebras of $A$.
\item\label{dcof} $\mathcal{P}(A)$ is $\ll$-approximately cofinal in $A^1_+$.
\item\label{<cof} $\mathcal{P}(A)$ is $\ll$-cofinal in $A^1_\mathrm{c}$.
\end{enumerate}
\end{proposition}

\begin{proof}\
\begin{itemize}
\item[\eqref{LU}$\Rightarrow$\eqref{auproj}]  If $A$ is LU then, for each finite $F\subseteq A$ and $\varepsilon>0$, we have $F\subseteq_\varepsilon B$ for some C*-subalgebra $B$ of $A$ with a unit $p_{F,\varepsilon}$.  Thus, for each $f\in F$,
\[\|fp_{F,\varepsilon}^\perp\|\leq\inf_{b\in B}(\|f-b\|\|p_{F,\varepsilon}^\perp\|+\|bp_{F,\varepsilon}^\perp\|)\leq\inf_{b\in B}\|f-b\|<\varepsilon.\]
So $(p_{F,\varepsilon})_{(F,\varepsilon)\in \Lambda}$ is an approximate unit, where $\prec$ on $\Lambda=[A]^{<\omega}\times(0,1)$ is given by
\[(F,\varepsilon)\prec(G,\delta)\quad\hbox{if and only if}\quad F\subseteq G\text{ and }\delta<\varepsilon.\]

\item[\eqref{auproj}$\Rightarrow$\eqref{LU}]  If  $(\Lambda, \prec)$ is a directed set and $A$ has an approximate unit 
$(p_\lambda)_{\lambda\in\Lambda}$ 
consisting of
projections,  then for each finite $F\subseteq A$ and $\varepsilon>0$, we have $\lambda$ with $\|f-fp_\lambda\|<\frac{1}{2}\varepsilon$ and $\|f-p_\lambda f\|=\|f^*-f^*p_\lambda\|<\frac{1}{2}\varepsilon$, for all $f\in F$.  Thus $\|f-p_\lambda fp_\lambda\|<\varepsilon$, for all $f\in F$, i.e. $F\subseteq_\varepsilon p_\lambda Ap_\lambda$.

\item[\eqref{auproj}$\Rightarrow$\eqref{dense}]  Immediate.

\item[\eqref{dense}$\Rightarrow$\eqref{dcof}]  Assuming \eqref{dense}, for each $a\in A^1_+$ and $\varepsilon>0$, we have $b$ in a C*-subalgebra $B$ of $A$ with a unit $p$ such that $\|a-b\|<\varepsilon$.  Thus $\|ap^\perp\|\leq\|a-b\|+\|bp^\perp\|<\varepsilon$.

\item[\eqref{dcof}$\Rightarrow$\eqref{auproj}]  See Proposition \ref{dcof=au}.

\item[\eqref{dcof}$\Rightarrow$\eqref{<cof}]  We argue as in \cite[V.3.2.17]{Blackadar2017}.  Specifically, take any $a\in A^1_\mathrm{c}$, so we have $b\in A^1_+$ with $a\ll b$.  Assuming $\mathcal{P}(A)$ is $\ll$-approximately cofinal in $A^1_+$, we have $p\in\mathcal{P}(A)$ with $\|\sqrt{b}p^\perp\|<1$ and hence $\|p^\perp bp^\perp\|<1$.  Thus the spectrum $S$ of $p^\perp b^\perp p^\perp$ and hence of $\sqrt{b^\perp}p^\perp\sqrt{b^\perp}$ is bounded away from $0$, i.e. $\sqrt{b^\perp}p^\perp\sqrt{b^\perp}$ is well-supported with support/range projection $r=f(\sqrt{b^\perp}p^\perp\sqrt{b^\perp})\in\widetilde{A}$, where $f(0)=0$ and $f(s)=1$ for all other $s\in S$.  As $ab^\perp=0$, $ar=0$ so $a\ll q$ for $q=r^\perp$.  Moreover, as $\pi(p^\perp)=\pi(w^\perp)=1$, we have $\pi(r)=1$ and hence $\pi(q)=0$, where $\pi$ is the unique homomorphism from $\widetilde{A}$ onto $\mathbb{C}$, i.e. $q\in A$.

\item[\eqref{<cof}$\Rightarrow$\eqref{dcof}]  For any $a\in A^1_+$ and $\varepsilon>0$, we have $b=(a-\varepsilon)_+\in A^1_\mathrm{c}$.  So if $\mathcal{P}(A)$ is $\ll$-cofinal in $A^1_\mathrm{c}$ then we have $p\in\mathcal{P}(A)$ with $b\ll p$ and hence $\|ap^\perp\|\leq\|a-b\|+\|bp^\perp\|\leq\varepsilon$, i.e. $\mathcal{P}(A)$ is $\ll$-approximately cofinal in $A^1_+$.\qedhere
\end{itemize}
\end{proof}

As in \eqref{auproj} above, we can prove the following

\begin{proposition}\label{au=ll-unit-of-projections} Let $A$ be a C*-algebra. 
$A$  is AU if and only if $A$ has a $\ll$-unit consisting of projections.
\end{proposition}

In fact, even a general $\ll$-unit makes an LU algebra an AU algebra.

\begin{theorem}\label{au=lu+waybelow}
Let $A$ be a C*-algebra. 
$A$  is AU if and only if $A$ is LU and $\ll$-unital.
\end{theorem}

\begin{proof}
The $\Rightarrow$ part is immediate.  Conversely, assume $A$ is LU and has a $\ll$-unit $U$.  We claim that
\[P=\{p\in\mathcal{P}(A):p\ll u\in U\}\]
is also a $\ll$-unit.  For this it suffices to show that $P$ is $\ll$-cofinal in $U$, i.e. for all $u\in U$, we have $p\in P$ with $u\ll p$.  By Proposition \ref{LUequivs}, we do at least have $p\in\mathcal{P}(A)$ with $u\ll p$.  As $U$ is a $\ll$-unit, we thus have $x\in U$ with $u\ll x$ and $\|px^\perp\|<1$ and hence $\|px^{\perp2}p\|<1$.  For $y=\sqrt{x^{\perp2\perp}}(\in C^*(x)\subseteq A)$, it follows that $py^2p$ and hence $ypy$ is well-supported with range/support projection $q$.  We have $z\in U$ with $x\ll z$ so $y\ll z$ and hence $q\ll z$, i.e. $q\in P$.  As $u\ll x,p$, we have $u\ll ypy$ and hence $u\ll q$, i.e. $P$ is indeed $\ll$-cofinal in $U$ and hence a $\ll$-unit.  Thus $A=\overline{\bigcup_{p\in P}pAp}$ is AU.
\end{proof}

Incidentally, while AU and LU algebras were not considered in \cite{Blackadar1994}, the dual property to LU was.  Specifically, in \cite{Blackadar1994}, $A$ was said to have \emph{property (SP)} (presumably for `Subalgebra Projections') if every non-zero hereditary C*-subalgebra contains a non-zero projection.  One can argue as in the proof of Proposition \ref{LUequivs} to obtain the following, 
where a subset $B$ of the unit sphere $A^{=1}_+$ is called \emph{$\ll$-approximately coinitial} if 
for every $a\in A^{=1}_+$ and $\varepsilon>0$  there is $b\in B$ with $b\ll_\varepsilon a$.

\begin{proposition}
For a C*-algebra $A$, the following are equivalent.
\begin{enumerate}
\item $A$ has property (SP).
\item $A^{=1}_+=\overline{\{a\in A^1_+:a\gg p\in\mathcal{P}(A)\setminus\{0\}\}}$.
\item $\mathcal{P}(A)$ is $\ll$-approximately coinitial in $A^{=1}_+$.
\item $\mathcal{P}(A)$ is $\ll$-coinitial in $\{a\in A^1_+:a\gg b\in A^1_+\setminus\{0\}\}$.
\end{enumerate}
\end{proposition}

\section{\texorpdfstring{$\omega_1$-unital C*-algebras}{omega\_1-unital C*-algebras}}

The goal of this section is to prove \autoref{main-omega1}.  The proof will rely on comparing $\mathbf{d}(a,b)=\|a-ab\|$ 
considered in Section 2.1 with another binary function $\mathbf{u}$ on $A$ which also is equal $0$
if and only if $a\ll b$. 
Specifically, let $\mathbf{u}$ measure how far away from $1$ we must choose a unitary in $\widetilde{A}$ if we want to have $u^*au\ll b$, i.e.
\[\mathbf{u}(a,b)=\inf\{\|1-u\|:u\in\widetilde{A},u^*u=1=uu^*,u^*au\ll b\}.\]
On $A^1_+$, $\mathbf{u}$ also satisfies the triangle inequality $\mathbf{u}(a,b)\leq\mathbf{u}(a,c)+\mathbf{u}(c,b)$ and
\[\mathbf{d}\leq2\mathbf{u}.\]
However, there is no way of conversely bounding $\mathbf{u}$ by any nice function of $\mathbf{d}$.  Indeed, if $A$ is commutative and $u$ is a unitary then $u^*au=a$ and hence $\mathbf{u}(a,b)=0$ if $a\ll b$ and $\infty=\inf\emptyset$ otherwise, even though $\mathbf{d}(a,b)$ can take arbitrarily small non-zero values.

To obtain the desired inequality, we need to consider a slightly bigger function than $\mathbf{d}$.  First consider the distance $\mathbf{h}(a,b)=||(a-b)_+||$ and define
\[a\leq_\delta b\qquad\Leftrightarrow\qquad\mathbf{h}(a,b)<\delta.\]
By \cite[Proposition 1.3]{bice-vignati}, $\frac{1}{2}\mathbf{h}\leq\mathbf{d}\leq\sqrt{\mathbf{d}\circ\mathbf{h}},\sqrt{\mathbf{h}\circ\mathbf{d}}$ so $\mathbf{d}$ has essentially the same magnitude as $\mathbf{d}\circ\mathbf{h}\circ\mathbf{d}$ near $0$.  Consider the bigger function $\ll\circ\ \mathbf{h}\ \circ\ll$ defined by
\[(\ll\circ\ \mathbf{h}\ \circ\ll)(a,b)=\inf\{\delta:c,d\in A^1_+,a\ll c\leq_\delta d\ll b\}.\]

\begin{theorem}\label{<<h<<}
For all $\varepsilon>0$ there is a $\delta>0$ such that, for all $a,b\in A^1_+$, if $(\ll\circ\ \mathbf{h}\ \circ\ll)(a,b)<\delta$, then $\mathbf{u}(a,b)<\varepsilon$.
\end{theorem}

\begin{proof}
Take $a,b,c,d\in A^1_+$ with $a\ll c\leq_\delta d\ll b$, for small $\delta>0$.  We will show that $a\ll_{O(\delta)}b$ for some function $O$ on $\mathbb{R}_+$ with $\lim_{r\rightarrow0}O(r)=0$.

First, by applying the continuous functional calculus, we replace $c$ with an element $f$ which has `smaller support'.  More precisely, take positive
\begin{equation}\label{delta'}
\delta'<\delta-\|(c-d)_+\|.
\end{equation}
and let $f=F(c)$, where $F:[0,1]\rightarrow[0,1]$ is defined by
\[F(r)=\begin{cases}0&\text{if }r\in[0,1-\delta']\\ (r-1+\delta')/\delta'&\text{if }r\in[1-\delta',1].\end{cases}\]
As $F(1)=1$ and $a\ll c$, we have $a\ll f$.  Also $f\leq c\leq_\delta d$ so $f\leq_\delta d$ and hence
\[a\ll f\leq_\delta d\ll b.\]

The smaller support of $f$ allows us to strengthen $f\leq_\delta d$ in several ways.  First, define $p\in A^{**}$ as the weak* limit of $f^{1/n}$, so $p$ is the support projection of $f$ in $A^{**}$.  As $||f^{1/n}-f^{1/n}c||=||F^{1/n}(c)-F^{1/n}(c)c||\leq\delta'$, for all $n$,
\begin{equation}\label{p-pc}
\|p-pc\|\leq\delta'.
\end{equation}
(alternatively, note $p$ is also the spectral projection of $c$ corresponding to the interval $[1-\delta',1]$).  By \cite[Proposition 1.3]{bice-vignati}, \eqref{delta'} and \eqref{p-pc},
\begin{equation}\label{||p-pd||}
\|p-pd\|^2\leq\|p-pc\|+\|(c-d)_+\|\leq\delta'+\|(c-d)_+\|<\delta.
\end{equation}
Letting $e=(1+\sqrt{d})^{-1}\in\widetilde{A}^1_+$, we have $(1-d)e=1-\sqrt{d}$ so \eqref{||p-pd||} yields
\begin{equation}\label{rootfrootd}
\|\sqrt{f}-\sqrt{d}\sqrt{f}\|=\|\sqrt{f}(1-\sqrt{d})\|=\|\sqrt{f}p(1-d)e\|\leq\|p(1-d)\|\leq\sqrt{\delta}.
\end{equation}
Again \eqref{||p-pd||} yields $\|p-pdp\|=\|p(1-d)p\|\leq\|p(1-d)\|<\sqrt{\delta}$.  As $pdp\ll p$, the C*-algebra they generate is isomorphic to $C(X)$ for some compact $X$ and we can consider $p$ as the constant function with value $1$ and $pdp$ as a function taking values $\geq1-\sqrt{\delta}$.  From this we see that $p\leq(1-\sqrt{\delta})^{-1}pdp$.  Multiplying by $\sqrt{f}$ on the left and right yields
\[f\leq(1-\sqrt{\delta})^{-1}\sqrt{f}d\sqrt{f}.\]

This allows us to apply the non-commutative Riesz decomposition from \cite[Proposition 1.4.10]{pedersen} with $x=\sqrt{f}$, $y=0$ and $z=\sqrt{1-\sqrt{\delta}}^{-1}\sqrt{f}\sqrt{d}$, obtaining $v\in A$ satisfying
\begin{equation}\label{vdef}
f=vv^*\qquad\text{and}\qquad v^*v\leq(1-\sqrt{\delta})^{-1}\sqrt{d}f\sqrt{d}.
\end{equation}
More explicitly, $v$ is the norm limit of $v_n=\sqrt{f}\sqrt{zz^*+1/n}^{-1}z$.\\
\textbf{Claim:} This implies that $v$ is close to $\sqrt{f}$.  For this first note that, for all $n\in\mathbb{N}$,
\begin{equation}\label{znorm}
\|\sqrt{zz^*+1/n}^{-1}z\|^2=\|z^*(zz^*+1/n)^{-1}z\|=\|z^*z(z^*z+1/n)^{-1}\|\leq1
\end{equation}
(as $F(zz^*)z=zF(z^*z)$ for any polynomial and hence any continuous $F$).  Letting $z_n=z^*\sqrt{zz^*+1/n}^{-1}z=z^*z\sqrt{z^*z+1/n}^{-1}\rightarrow\sqrt{z^*z}$, we have
\begin{align*}
\|v-\sqrt{z^*z}\|&=\lim\|v_n-z_n\|\\
&=\lim\|\sqrt{f}\sqrt{zz^*+1/n}^{-1}z-z^*\sqrt{zz^*+1/n}^{-1}z\|\\
&\leq\|\sqrt{f}-z^*\|\qquad\text{by \eqref{znorm}}\\
&=\|\sqrt{f}-\sqrt{1-\sqrt{\delta}}^{-1}\sqrt{d}\sqrt{f}\|\\
&\leq\|\sqrt{f}-\sqrt{d}\sqrt{f}\|+\sqrt{1-\sqrt{\delta}}^{-1}-1.\\
&\leq \sqrt{\delta}+\sqrt{1-\sqrt{\delta}}^{-1}-1\qquad\text{by \eqref{rootfrootd}}.
\end{align*}
Thus $\|v-\sqrt{z^*z}\|\leq\|\sqrt{f}-z^*\|<O(\delta)$, for some function $O$ with $\lim_{r\rightarrow0}O(r)=0$.  But on bounded subsets, taking adjoints, multiples and square roots is uniformly continuous so $\|\sqrt{f}-z^*\|<O(\delta)$ implies that, for another such function $O'$, we have $\|\sqrt{z^*z}-\sqrt{f}\|=\|\sqrt{z^*z}-\sqrt{\sqrt{f}\sqrt{f}^*}\|<O'(\delta)$.  Thus
\[\|v-\sqrt{f}\|\leq\|v-\sqrt{z^*z}\|-\|\sqrt{z^*z}-\sqrt{f}\|\leq O(\delta)+O'(\delta),\]
proving the claim.

Define $w\in\widetilde{A}$ by
\begin{equation}\label{wdef}
w=v+\sqrt{1-\sqrt{vv^*}}\sqrt{1-\sqrt{v^*v}}.
\end{equation}
Note that if we replaced $v$ above with $\sqrt{f}$ (or any other element of $A^1_+$) we would have $w=1$.  So again by the uniform continuity on $\widetilde{A}^1$ of all the functions involved, $\|v-\sqrt{f}\|<O(\delta)+O'(\delta)$ implies that $\|1-w\|<O''(\delta)$, for another function $O''$ with $\lim_{r\rightarrow0}O''(r)=0$.  In particular, $w$ is invertible so we may let $u\in\widetilde{A}$ be the unitary in the polar decomposition of $w$,
\begin{equation}\label{udef}
u=\sqrt{ww^*}^{-1}w.
\end{equation}
Again, as $\|1-w\|<O''(\delta)$, $\|1-u\|<O'''(\delta)$ for another function $O'''$ with $\lim_{r\rightarrow0}O'''(r)=0$.  In particular, for any $\varepsilon>0$ we could choose $\delta>0$ such that $O'''(\delta)<\varepsilon$ and hence $\|1-u\|<\varepsilon$.

Now let $q,r\in A^{**}$ be the weak* limits of $(vv^*)^n$ and $(v^*v)^n$ respectively, so
\begin{equation}\label{qv=vr}
qv=\lim_n(vv^*)^nv=\lim_nv(v^*v)^n=vr.
\end{equation}
Note $q^\perp$ and $r^\perp$ are the support projections of $1-vv^*$ and $1-v^*v$ respectively so
\begin{equation}\label{qandr}
q\sqrt{1-\sqrt{vv^*}}=0=\sqrt{1-\sqrt{v^*v}}r.
\end{equation}
By \eqref{wdef}, \eqref{qv=vr} and \eqref{qandr}, $qw=qv=vr=wr$ and hence $rw^*=w^*q$.  Thus $qww^*=wrw^*=ww^*q$ so $q\sqrt{ww^*}^{-1}=\sqrt{ww^*}^{-1}q$ and hence, by \eqref{udef},
\begin{equation}\label{qu=ur}
qu=q\sqrt{ww^*}^{-1}w=\sqrt{ww^*}^{-1}qw=ur.
\end{equation}
Further note $q$ is the spectral projection of $vv^*$ corresponding to $1$ so $q\ll vv^*$ and $q$ is maximal for this property, i.e. $x\ll q$, for any $x\ll vv^*$.  In particular, as $a\ll f=vv^*$, by \eqref{vdef}, we have
\begin{equation}\label{allq}
a\ll q.
\end{equation}
Likewise, $r\ll v^*v$.  As $d\ll b$, by assumption, $(1-\sqrt{\delta})^{-1}d\ll b$.  Combined with \eqref{vdef}, we thus have
\[r\ll v^*v\leq(1-\sqrt{\delta})^{-1}d\ll b.\]
Thus $r\ll v^*v\ll b$ (because $x\leq y\ll z$ implies $x\ll z$ and, more generally, $||x-xz||^2\leq||(x-y)_+||+||y-yz||$, by \cite[Proposition 1.3]{bice-vignati}) and hence
\begin{equation}\label{rllb}
r\ll b.
\end{equation}
Using \eqref{qu=ur}, \eqref{allq} and \eqref{rllb}, we thus obtain
\[aub=aqub=aurb=aur=aqu=au.\]
Multiplying by $u^*$ yields $u^*aub=u^*au$, i.e. $u^*au\ll b$, as required.
\end{proof}

What we actually need is the same result for
\[(\ll\circ\ \mathbf{d}\ \circ\ll)(a,b)=\inf\{\delta:c,d\in A^1_+,a\ll c\ll_\delta d\ll b\}.\]

\begin{corollary}\label{<<d<<}
For all $\varepsilon>0$ there is a $\delta>0$ such that, for all $a,b\in A^1_+$, if $(\ll\circ\ \mathbf{d}\ \circ\ll)(a,b)<\delta$, then $\mathbf{u}(a,b)<\varepsilon$.
\end{corollary}

\begin{proof}
Given $\epsilon>0$ take $\delta>0$ as in \autoref{<<h<<}.  By \cite[Proposition 1.3]{bice-vignati}, $\mathbf{h}\leq2\mathbf{d}$ so $(\ll\circ\ \mathbf{d}\ \circ\ll)(a,b)<\frac{1}{2}\delta$ implies $(\ll\circ\ \mathbf{h}\ \circ\ll)(a,b)<\delta$ and hence $\mathbf{u}(a,b)<\varepsilon$.
\end{proof}

\begin{corollary}\label{SigmaExtension}
If $A$ is $\sigma$-unital C*-algebra, then for every countable $\ll$-directed $B\subseteq A^1_+$ 
there is a countable $\ll$-unit $U$ for $A$ such that $B\subseteq U$.
\end{corollary}

\begin{proof}
Corollary \ref{<<d<<} allows us to follow the proof of \cite[Lemma 5.3]{farah-katsura}.  First, let $(b_n)_{n\in \N}$ be a
 $\ll$-increasing enumeration of a $\ll$-cofinal subset of $B$.  As $A$ is $\sigma$-unital, we also have a $\ll$-increasing approximate unit $(a_n)_{n\in \N}$ by Proposition \ref{sigma-unital}.  So 
 by Proposition \ref{cf+dir=inc} for any $\varepsilon>0$, we have $n_1$ with
\[b_2\ll_\varepsilon a_{n_1}\]
As long as $\varepsilon>0$ is sufficiently small, Corollary \ref{<<d<<} yields a unitary $u_1\in\widetilde{A}$ with
\[u_1^*b_1u_1\ll a_{n_1+1}\qquad\text{and}\qquad\|1-u_1\|<1/2.\]
Then again, for any $\varepsilon>0$, we have $n_2>n_1$ with
\[u_1^*b_3u_1\ll_\varepsilon a_{n_2}.\]
Again, as long as $\varepsilon$ is sufficiently small, Corollary \ref{<<d<<} yields a unitary $u_2\in\widetilde{A}$ with
\[u_2^*u_1^*b_2u_1u_2\ll a_{n_2+1}\qquad\text{and}\qquad\|1-u_2\|<1/4.\]
In fact, rather than applying Corollary \ref{<<d<<} to the entirety of $\widetilde{A}$, we can apply it to the C*-subalgebra $B$ generated by $u_1^*b_2u_1$, $a_{n_1+1}$ and $1$ to obtain $u_2\in B$.  As $u_1^*b_1u_1$ is $\ll$-below $u_1^*b_2u_1$, $a_{n_1+1}$ and $1$, $u_1^*b_1u_1$ commutes with all the generators and hence with all $b\in B$.  In particular, $u_1^*b_1u_1$ commutes with $u_2$ so
\[u_2^*u_1^*b_1u_1u_2=u_1^*b_1u_1u_2^*u_2=u_1^*b_1u_1.\]
Continuing, we obtain a sequence of unitaries $(u_n)_{n\in \N}\subseteq\widetilde{A}$ with a convergent product
\[u=u_1u_2\ldots\]
such that $u^*b_ku=u_k^*b_ku_k\ll a_{n_k+1}$, for all $k\in\mathbb{N}$.  Thus
\[U=B\cup\{ua_nu^*:n\in\mathbb{N}\}\]
is the required $\ll$-unit  and an extension of $B$.
\end{proof}

\begin{corollary}
If $A$ is $\omega_1$-unital then $A$ is $\ll$-unital.
\end{corollary}

\begin{proof}
Let $(a_\alpha)_{\alpha<\omega_1}$ enumerate $\ll$-approximately cofinal subset of $A$,
which exists by Proposition \ref{dcof=au}.  For each $\beta<\omega_1$, let $A_\beta$ be the C*-subalgebra of $A$ generated by $(a_\alpha)_{\alpha<\beta}$.  As each $A_\beta$ is separable and hence $\sigma$-unital, Corollary \ref{SigmaExtension} yields a $\subseteq$-increasing $\omega_1$-sequence of countable $\ll$-directed subsets $(U_\beta)_{\beta<\omega_1}$ such that each $U_\beta$ is
 $\ll$-approximately cofinal in $A_\beta$, for every $\beta<\omega_1$.  Thus $U=\bigcup_{\beta<\omega_1}U_\beta$ is $\ll$-directed and $\ll$-approximately cofinal in $A_{\omega_1}$ and hence in $A$, so
 $A$ is $\ll$-unital by Proposition \ref{cf+dir=inc}.
\end{proof}

\section{\texorpdfstring{C*-algebras without $\ll$-Units}{C*-algebras without <<-Units}}\label{counterexamples}

Recall the C*-algebra $D$ we defined in \eqref{Ddef} and the 
projections $p=(p_n)_{n\in \N}$ and $q=(q_n)_{n\in\N}$ in $D$ from Lemma \ref{no-way-above}. 
 Here we show that there are non-$\ll$-unital C*-sublagebras of $\ell_\infty^X(D)$, for certain sets $X$.

\begin{definition}\label{def-family-algebra}
Suppose $X$ is a set and  $\F=\{A_i, B_i: i\in \I\}$ is a family of subsets of $X$ satisfying $A_i\cap B_i=\emptyset$, 
for each $i\in \I$.  Define a projection $\PP_i\in \ell_\infty^X(D)$ by
\[\PP_i(t)=\begin{cases}
p&\text{if }t\in A_i, \\
q&\text{if }t\in B_i,\\
0&\text{otherwise}.\end{cases}\]
$A_\F$ is the C*-algebra generated by $(\PP_i)_{i\in \I}$ in $\ell_\infty^X(D)$.
\end{definition}

\begin{lemma}\label{family-algebra} Suppose  that $\F=\{A_i, B_i: i\in \I\}$ is
a family of subsets of an $X$ satisfying $A_i\cap B_i=\emptyset$ 
for each $i\in \I$ and  such that for every $F: \I\rightarrow [\I]^{\leq \omega}$
there are distinct $i, i'\in \I$ satisfying 
\[\tag{$*$}\label{star}(A_i\cap B_{i'})\setminus \bigcup\{A_j\cup B_j:  j\in (F(i)\setminus \{i\})\cup (F(i')\setminus \{i'\})\}\not=\emptyset.\]
Then $A_\F$ is a  $2$-subhomogenous C*-subalgebra of
$\B(\ell_2(X))$ of density not bigger than $|\F|$ without a $\ll$-unit.
\end{lemma}

\begin{proof}
Suppose that $A_\F$ had a $\ll$-unit $U$.  In particular, for every $i\in \I$, 
we have $u_i\in U$ with $\|\PP_i-\PP_i u_i\|<1$.  For each $i\in \I$ we also have countable 
$F(i)\subseteq \I$ such that $u_i$ is in the C*-subalgebra of
$\A_\F$ generated by $(\PP_{j})_{j\in F(i)}$.
By the hypothesis there are distinct $i, i'\in \I$ and  $x\in X$ which belongs to the set in (*). 
Thus, for all $j\in (F(i)\setminus \{i\})\cup (F(i')\setminus \{i'\})$, we have $\PP_{j}(x)=0$ and hence
$u_i(x)\in\mathbb{C}p$ and  $u_{i'}(x)\in\mathbb{C}q$.  

 As $U$ is $\ll$-directed, 
we have $u\in U$ with $u_i,u_{i'}\ll u$ and so  $u_{i}(x)=zp\ll u(x)$ and $u_{i'}(x)=z'q\ll u(x)$,
 for some $z,z'\in\mathbb{C}$.  As $\|\PP_i-\PP_iu_i\|,\|\PP_{i'}-\PP_{i'}u_{i'}\|<1$, 
we have $z,z'\not=0$ and hence $p,q\ll u(x)$, contradicting Lemma \ref{no-way-above}.  Thus $A$ is not $\ll$-unital.
\end{proof}

\begin{example}\label{ad-algebra}
A ZFC construction of a $2$-subhomogeneous C*-subalgebra of $\B(\ell_2(\omega_1))$ of density $\omega_2$
without a $\ll$-unit.
\end{example}
%\begin{proof}
\noindent{\it Construction.}
It is enough to construct a family $\F=\{A_\xi, B_\xi: \xi<\omega_2\}$ of subsets of $\omega_1$
satysfying (*) of Lemma \ref{family-algebra}.  We do it by recursion on $\xi<\omega_2$
requiring that all $A_\xi$s and $B_\xi$s are uncountable and that
$(A_\xi\cup B_\xi)\cap (A_{\xi'}\cup B_{\xi'})$ are countable for all $\xi'<\xi<\omega_2$.
For $\xi<\omega_1$ we choose all the above sets pairwise disjoint.
Now suppose that we are done till $\omega_1\leq \xi<\omega_2$. Renumerate all $\{A_\eta, B_\eta: \eta<\xi\}$
as $\{A^\gamma, B^\gamma: \gamma<\omega_1\}$  and  by recursion on $\gamma<\omega_1$  pick distinct
\begin{itemize}
\item $\alpha_\gamma\in B^\gamma\setminus\bigcup_{\delta<\gamma}
(A^\delta\cup B^\delta\cup \{\alpha_\delta, \beta_\delta\}\cup \gamma)$
\item $\beta_\gamma\in A^\gamma\setminus\bigcup_{\delta<\gamma}
(A^\delta\cup B^\delta\cup \{\alpha_\delta, \beta_\delta\}\cup\gamma)$
\end{itemize}
This can be done as the above sets are uncountable due to the recursive hypothesis about the countable
interesections and uncountable sets. Putting $A_\xi=\{\alpha^\gamma: \gamma<\omega_1\}$ 
and $B_\xi=\{\beta^\gamma: \gamma<\omega_1\}$ preserves
this recursive hypothesis. This completes the construction  of $\F$.

Now suppose that $F: \omega_2\rightarrow [\omega_2]^{\leq\omega}$.  For each $\xi<\omega_2$ let
$\theta_\xi<\omega_1$ be such that $\bigcup\{(A_\xi\cup B_\xi)\cap 
(A_{\beta}\cup B_{\beta}): \beta\in F(\xi)\setminus\{\xi\}\}\subseteq \theta_\xi$.
It can be found since the intersections are countable. Let $\theta<\omega_1$ be such that 
$\Omega_\theta=\{\xi: \theta_\xi=\theta\}$
has cardinality $\omega_2$.  Let $\xi\in \Omega_\theta$ be such that $\Omega_\theta\cap\xi$ is uncountable.
Then when constructing $A_\xi, B_\xi$ there was $\gamma$ with $\theta<\gamma<\omega_1$
such that $(A^\gamma, B^\gamma)=(A_\eta, B_\eta)$  and
$\eta\in \Omega_\theta\cap \xi$. It follows from the construction that $\alpha_\gamma\in (A_\xi\cap B_\eta)\setminus\theta$,
which means that 
$$\alpha_\gamma\in (A_\xi\cap B_\eta)\setminus \bigcup\{A_\beta\cup B_\beta: \beta\in (F(\xi)\setminus \{\xi\})\cup
 (F(\eta)\setminus \{\eta\})\},$$
as required in Lemma \ref{family-algebra}\hfill$\Box$
%\end{proof}
\vskip 13pt
The remaining examples of C*-algebras without $\ll$-units which we present are scattered, and so we will
need a couple of  simple lemmas concerning such C*-algebras and scattered compact spaces.

\begin{lemma}\label{subscattered} Suppose that $A$ is a  scattered C*-algebra and $K$ is a scattered compact 
Hausdorff space. Then the C*-algebra $C(K, A)$ is scattered.
\end{lemma}
\begin{proof} We will use the characterization of scattered C*-algebras due to Wojtaszczyk 
(\cite{wojtaszczyk}, cf. Theorem 1.4 (6) of \cite{cantor-bendixson}) as the C*-algebras where all
self-adjoint elements have countable spectrum.  Let $f\in C(K, A)$ be self-adjoint.  The range of $f$ is
a compact scattered metric space (as a subspace of $A$) which is therefore second countable and hence
 countable by the Cantor-Bendixson theorem (1.7.10 of \cite{engelking}). As each value of $f$
has countable spectrum and $K$ is compact, the spectrum of $f$ must be countable as the union of the spectra
of its values.
\end{proof}

Recall that a partial order is said to be \emph{well-founded} if it has no infinite strictly decreasing chains.
Also, by a \emph{$\wedge$-subsemilattice} of a Boolean algebra, we mean a subset which is closed under taking
pairwise infima $\wedge$. A Boolean algebra is said to be \emph{superatomic} if every quotient of it has an atom,
i.e. an element $a>0$ such that $a\geq b\geq0$ implies $b=a$ or $b=0$.  The
class of superatomic Boolean algebras is very well investigated as they are exactly the Boolean algebras arising from the clopen subsets of scattered compact spaces (see e.g. \cite{roitman}, cf. \cite{cantor-bendixson}).

\begin{definition}
We call a Boolean algebra
 \emph{well-generated} if it is generated by a well-founded $\wedge$-subsemilattice.
\end{definition}

\begin{lemma}\label{well-generated}
Every well-generated Boolean algebra is superatomic.
\end{lemma}

\begin{proof}
First we show that quotients of well-generated Boolean algebras are well-generated.  
Say $\pi:A\rightarrow B$ is a Boolean homomorphism and $G\subseteq A$ 
is a well-founded $\wedge$-subsemilattice generating $A$.  
Then $\pi[G]$ generates $\pi[B]$.  If $\pi[G]$ were not well-founded,
 then we would have $(a_n)\subseteq G$ such that $(\pi(a_n))_{n\in\mathbb{N}}$ is strictly decreasing. 
 Then $a'_n=a_1\wedge\cdots\wedge a_n$ would be a strictly decreasing sequence in $G$, contradicting well-foundedness.

To show that well-generated Boolean algebras are superatomic, it thus suffices to show that they always contain an atom.  So say $G$ is a well-founded $\wedge$-subsemilattice generating a Boolean algebra $A$.  If $A$ is the $2$-element Boolean algebra then its top element $1$ is an atom.  Otherwise, $G$ must contain a non-zero element $g$, which we can take to be minimal by well-foundedness.  Let
\[B=\{a\in A:a\wedge g=g\text{ or }0\}.\]
By the minimality of $g$, $B$ contains $G$.  But $B$ is also closed under 
all the Boolean operations $\wedge$, $\vee$ and $^c$ so, as $G$ generates $A$, we must have $A=B$ and hence $g$ is an atom of $A$.
\end{proof}

In fact  \cite{bonnet-robert} contains a slightly different definition of being well-generated 
which refers to lattices rather $\wedge$-semilattices.  
However, the lattice generated by any well-founded 
$\wedge$-subsemilattice within a Boolean algebra is again well-founded,
 so our definition is equivalent (we leave the proof as an exercise).  The above paper
also contains a result
{\cite[Propositon 2.7]{bonnet-robert}} saying that well-generated Boolean algebras are superatomic which
is then equivalent to our Lemma \ref{well-generated}.
We note in passing that not all superatomic Boolean algebras are well-generated 
\textendash\, see \cite[Theorem 3.4]{bonnet-robert}.

%\begin{lemma}\label{chain} Suppose that $A$ is a Boolean algebra of subsets of an $X$ generated
%by  a well-ordered with respect to inclusion chain of subsets of $X$ and by a family of singletons of $X$.
%Then $A$ is superatomic i.e., its  Stone space is a scattered compact space.
%\end{lemma}
%\begin{proof}
%We may assume that the generators which are singletons correspond to
%all singletons of $X$  as subalgebras of superatomic Boolean algebras (algebras of clopen subsets of scattered spaces,
%cf. Section 2 of \cite{cantor-bendixson}) are superatomic.  Consider the quotient of  $A$ by the ideal of
%finite sets. The resulting algebra is generated by a well-ordered chain as  an infinite decreasing sequence
%of images of the generators  under the quotient map would give rise to an infinte decreasing sequence
%of the generators. So the quotient is isomorphic to the Boolean algebra of all clopen subsets of
%an ordinal, i.e,. it is superatomic. It follows that the original
%algebra is superatomic because the ideal generated by singletons generates
%a superatomic algebra (cf. Theorem 2.2. of \cite{cantor-bendixson}).
%\end{proof}

\begin{example}\label{crosses-algebra}
A ZFC construction of a $2$-subhomogeneous scattered C*-subalgebra of $\B(\ell_2(\omega_2))$ of density $\omega_2$
without a $\ll$-unit.
\end{example}
\noindent{\it Construction.}
We will construct an appropriate family $\F=\{A_\xi, B_\xi: \xi<\omega_2\}$ of subsets of 
$X=\{(\xi, \eta): \xi<\eta<\omega_2\}$.
satysfying (*) of Lemma \ref{family-algebra}. This construction is simillar to \cite[Example 2.1]{commuting-aus}\footnote{We would like to thank I. Farah for drawing our attention to \cite{commuting-aus}.}
modulo the use of the free set lemma like in \cite{farah-katsura} which allows us to lower $2^{2^\omega}$ to $\omega_2$.  Put 
$$A_\xi=\{(\xi, \eta): \xi<\eta<\omega_2\}, \   B_\xi=\{(\eta, \xi): \eta<\xi<\omega_2\}.$$
Suppose $F:[\omega_2]\rightarrow [\omega_2]^{\leq\omega}$. 
 Take $\xi_0<\eta_0<\omega_2$ such that $\eta_0\not\in F(\xi_0)$ nor $\xi_0\not\in F(\eta_0)$.
This follows from a stronger result called 
the Hajnal free set lemma (Lemma 19.1 in  \cite{hajnal}) but in this case is very simple (cf. Lemma 2.1. of \cite{farah-katsura}).
Then $(\xi_0, \eta_0)\in A_{\xi_0}\cap B_{\eta_0}$. Moreover $(\xi_0, \eta_0)\in A_\xi$ implies $\xi_0=\xi$
as well as $(\xi_0, \eta_0)\in B_\eta$ implies $\eta_0=\eta$, so $(\xi_0, \eta_0)$ is not in any
$A_\xi, B_\eta$ for $\xi, \eta\in (F(\xi_0)\setminus\{\xi_0\})\cup (F(\eta_0)\setminus\{\eta_0\})$.
It follows that the family $\F$ satisfies \eqref{star} of Lemma \ref{family-algebra} and so $\A_\F$ has no $\ll$-unit.

To prove that $A_\F$ is scattered, first note that  it can be naturally embedded into the algebra $C(K, D)$
where $K$ is the Stone space of the Boolean algebra $B$ of subsets of $X$ generated by $\{A_\xi, B_\xi: \xi<\omega_2\}$.
Note that the elements of this family are either disjoint or have a singleton as the intersection and all singletons
appear as such intersections.  Together with the empty set, they form a well-founded $\wedge$-semilattice (where the singletons have rank $1$ and each $A_\xi$ and $B_\xi$ has rank $2$).  Thus $B$ is well-generated and hence superatomic, by Lemma \ref{well-generated}.  Now Lemma \ref{subscattered} implies that $A_\F$ is scattered. \hfill$\Box$
\vskip 13pt
The following  example will be based on trees. See Section 2.3. for our terminology concerning trees.

\begin{lemma}\label{tree-algebra} Suppose that $\kappa$ is a cardinal and $T$ is a tree
of cardinality and height $\kappa$ satisfying
\begin{enumerate}
\item  $\omega<cf(\kappa)$, 
\item $\kappa<|Br_\kappa(T)|$.
%\item $sup(\{|Lev_\alpha(T)|: \alpha<\kappa\})<|Br_\kappa(T)|$.
\end{enumerate} 
Then $\F(T)=\{A_b, B_b: b\in Br_\kappa(T)\}$ satisfies the hypothesis of Lemma \ref{family-algebra}
for $X=T$ and $A_b=b$, $B_b=b^\#$. Moreover
$A_{\F(T)}$ is a scattered C*-algebra.
\end{lemma}

\begin{proof}
Let $F: Br_\kappa(T)\rightarrow [Br_\kappa(T)]^{\leq \omega}$.
 As $cf(\kappa)>\omega$, we can define $\alpha_b<\kappa$ by
\[\alpha_b=\sup_{b'\in F(b)\setminus\{b\}}\min\{\beta<\kappa:b'\cap Lev_\beta(T)\not=b\cap Lev_\beta(T)\}.\]
So, for all $b'\in X_b\setminus\{b\}$ we have $b'\cap Lev_{\alpha_b}(T)\neq b\cap Lev_{\alpha_b}(T)$. 
 
%Put $\theta=sup(\{|Lev_\alpha(T)|: \alpha<\kappa\})$. 
As for all $b\in Br_\kappa(T)$ we have $\alpha_b<\kappa<|Br_\kappa(T)|$, 
 there are $X\subseteq Br_\kappa(T)$ and $\alpha<\kappa$  such that $\kappa<|X|$ 
 and $\alpha_b=\alpha$, for all $b\in X$. 
As $Lev_\alpha(T)\leq|T|=\kappa$ we have distinct $b,b'\in X$ 
such that $b\cap Lev_\alpha(T)=b'\cap Lev_\alpha(T)$.  
Let $\beta\in (\alpha, \kappa)$ be the minimal ordinal such that
 $b\cap Lev_{\beta}(T)\not=b'\cap Lev_\beta(T)$, so if 
$\{t\}=b'\cap Lev_\beta(T)$, 
then $t\in b^\#$, in other words $t$ belongs to
$$(A_b\cap B_{b'})\setminus \bigcup\{A_c\cup B_c:  c\in (F(b)\setminus \{b\})\cup (F(b')\setminus \{b'\})\},$$
as required in Lemma \ref{family-algebra}.

Now we prove that $A_{\F(T)}$ is scattered.  $A_{\F(T)}$ can be considered  as a subalgebra
of $C(K,D)$ where $K$ is the Stone space of the Boolean algebra $B$ generated in $\wp(T)$ by 
the family of sets $\{b, b^\#: b\in Br_\kappa(T)\}$.  The intersection of a branch and an exit-set contains at most one element, while the intersection of two branches is an initial segment of the tree.  As the $\wedge$-semilattice of these intersections is well-founded, $B$ is well-generated and hence superatomic, by Lemma \ref{well-generated}, and hence $K$ is scattered, by Lemma \ref{subscattered}.
\end{proof}

\begin{example}\label{ex-tree-zfc} A ZFC construction of a $2$-subhomogeneous scattered C*-subalgebra of $\B(\ell_2(2^\omega))$ of density $2^\kappa$, for $\kappa$ minimal with $2^\omega<2^\kappa$, without a $\ll$-unit.
\end{example}
\noindent{\it Construction.} Consider the tree $\{0,1\}^{<\kappa}$. By Lemma \ref{minimal-cardinal}
the hypothesis of Lemma \ref{tree-algebra} is satisfied.\hfill$\Box$
\vskip 13pt

\begin{example} \label{ex-tree-can}
A $2$-subhomogeneous scattered C*-subalgebra of $\B(\ell_2(\omega_1))$ of density 
bigger than $\omega_1$ without a $\ll$-unit, assuming that a Canadian tree exists.
\end{example}
\noindent{\it Construction.} Consider a Canadian tree $T$ i.e. of height and
cardinality $\omega_1$ but with more than $\omega_1$ uncountable branches.  By Lemma \ref{minimal-cardinal}
the hypothesis of Lemma \ref{tree-algebra} is satisfied.\hfill$\Box$
\vskip 13pt
\begin{proposition}\label{extension} Suppose that $\F$ is a family satisfying Lemma \ref{family-algebra} 
and that the C*-algebra $A_\F$ is scattered. Then the
 algebra $A_\F$ is an extension of an AF C*-algebra by an AF C*-algebra.
\end{proposition}
\begin{proof}
 Let 
 $$J=\{a\in \ell_\infty^X(D): \lim_{n\rightarrow\infty} a(x)(n)=0\ \hbox{for all}\  x\in X\},  \ \ I=A_\F\cap J.$$
As in  \cite[Remark 3.1.3]{murphy} we have that $(A_\F+J)/J$ and $A_\F/I$ are isomorphic.
But $\ell_\infty^X(D)/J$ is commutative as the homomorphism $h$ sending $a\in\ell_\infty^X(D)$
to $h(a)\in \ell_\infty^X(\C)$ given by $h(a)(x)=z$ if an only if $\lim_{n\rightarrow\infty}a(x)(n)=zP_0$  has
its kernel equal to $J$. It follows that $(A_\F+J)/J$ and so $A_\F/I$ are commutative as well.
Quotients of scattered C*-algebras are scattered and scattered commutative C*-algebras are
of the form $C_0(K)$ for $K$ locally compact and scattered, that
is where every nonempty subset has an isolated point (cf. eg. \cite[2.7, 1.4]{cantor-bendixson}).
But locally compact scattered spaces are totally disconnected by (\cite[6.1. 23]{engelking}, 
cf. \cite[Theorem 2.2, Lemma 2.5]{cantor-bendixson}).
Now to conlude that $A_\F/I$ is an AF C*-algebra it is enough to note that  the C*-algebras generated
by characteristic functions of elements of finite Boolean subalgebras of the Boolean algebra $Clop(K)$ of clopen subsets of $K$ form
a directed family of finite-dimensional algebras whose union is dense in $C_0(K)$.

So we are left with proving that $I$ is an AF C*-algebra. Let $B\subseteq A_\F$ denote the collection
of all elements $a$ of $A_\F$ for which the set $\{a(x): x\in X\}$ is finite. As the generators $\PP_i$ 
for $i\in \I$
belong to $B$ and the C*-algebra operations leave $B$ invariant, $B$ is norm dense in $A_\F$. 
We will also consider the set $C\subseteq B\cap I$ of all elements $a$ of $B$ such that
for each $x\in X$ we have $a(x)(n)=0$ for all but finitelly many $n\in \N$.
  It is clear that  that
 the C*-algebras generated by finitelly many elements of $C$ are finite-dimensional. 
 We will prove that any element of $I$ can be approximated by an element of $C$, which will
 imply that the finite-dimensional subalgebras of $C$ form a directed family whose union is dense in $I$.

Let  $a\in I$  and $\varepsilon>0$, take an $a_1\in B$ with $\|a-a_1\|<\varepsilon/3$.
We have 
$$\|{{a+a^*}\over 2}-{{a_1+a_1^*}\over 2}\|,  \|{{a-a^*}\over 2i}-{{a_1-a_1^*}\over 2i}\|<\varepsilon/3$$
and both 
$${{a_1+a_1^*}\over 2}\  \hbox{and}\  {{a_1-a_1^*}\over 2i}$$
 are in $B$. 
So we may assume that both $a$ and $a_1$ are selfadjoint, and finding $f\in C_0(\sigma(a_1))$
such that $f(a_1)\in C$ with $\|a-f(a_1)\|<\varepsilon$ will take care of a general case of $a$.

As $\lim_{n\rightarrow\infty}a(x)(n)=0$ for each $x\in X$ we have that $\|a_1(x)(n)\|<\varepsilon/3$ for all but finitely many
$n\in \N$ and all $t$. Using the fact that self adjoint elements in scattered C*-algebras 
have countable spectra (\cite[1.4 (6)]{cantor-bendixson}) find $f\in C_0(\sigma(a_1))$  such that
$f|[-r, r]=0$ and $f(t)=t$ for $t\in \sigma(a_1)\setminus[-r, r]$ for some $r\in (\varepsilon/3, 2\varepsilon/3)\setminus
\sigma(a_1)$.

We have $\|f(a_1)-a_1\|<2\varepsilon/3$ and so $\|a-f(a_1)\|<\varepsilon$.
On the other hand $f(a_1)\in B$  since $f(a_1)(t)=f(a_1(t))$ and $a_1\in B$. 
Also if $\|a_1(x)(n)\|<\varepsilon/3$, then $f(a_1)(x)(n)=0$ as $f$ is applied coordinatewise
and $\sigma(a_1(x)(n))\subseteq [\varepsilon/3, \varepsilon/3]$ for such $n\in \N$.
But $\|a_1(x)(n)\|<\varepsilon/3$ for all but finitely many $n\in \N$ for each $x\in X$
as $\|a-a_1\|<\varepsilon/3$ and $\lim_{n\rightarrow\infty}a(x)(n)=0$ for all $x\in X$
since $a\in I$. This completes the construction of the required element of
$C$ which approximates $a\in I$.

\end{proof}

Now we show that consistently there are C*-subalgebras
of $\B(\ell_2)$ which are not $\ll$-unital.

\begin{definition}\label{separated}
$\mathcal S\subseteq \wp(\N)$ is called separated if for every $\mathcal X\subseteq \mathcal S$
 there is $Z_{\mathcal{X}}\subseteq \N$ such that $X\setminus Z_{\mathcal{X}}$ is finite
if $X\in\mathcal{X}$ and $X\cap Z_{\mathcal{X}}$ is finite if $X\in\mathcal{S}\setminus\mathcal{X}$.
\end{definition}

It is consistent with ZFC to have uncountable separated families, even in the presence of a Kurepa tree. 
 Indeed, by \cite[Theorem 3.9]{hrusak} or \cite[Propositions 2.1. and 2.2]{hrusak-fernando}, 
there are uncountable separated families if and only if there are uncountable $Q$-sets,
 i.e. subsets $Y$ of $\mathbb{R}$ such that every $Y'\subseteq Y$ is a
 relative $G_\delta$-set of $Y$.  Starting with the constructible
 universe $L$, we can use a c.c.c. forcing to obtain MA$+\neg$CH.  
In particular, cardinals and hence the Kurepa tree in $L$ is preserved, 
while all subsets of $\R$ of cardinality $\omega_1$ become $Q$-sets (see \cite[Theorem 4.2]{miller}).

Note that the existence of separated families of size $\omega_1$ implies $2^\omega=2^{\omega_1}$ (in particular, CH fails).

\begin{example}\label{Bell2}
 A $2$-subhomogeneus C*-subalgebra of  $\B(\ell_2)$ which is LF but not $\ll$-unital
under the hypothesis of the existence of an uncountable $Q$-set.
\end{example}

\noindent{\it Construction.}
As explained above the existence of an uncountable $Q$-set is equivalent to the existence
of an uncountable separated family.
We may assume that our separated family $\{S_\alpha: \alpha<\omega_1\}\subseteq \wp(\N)$ has cardinality $\omega_1$. 
Let $\F=\{A_\xi, B_\xi: \xi<\omega_2\}$ be as in Example \ref{ad-algebra}. 
 For any $\xi<\omega_2$ we have disjoint $Y_\xi, Z_\xi\subseteq \N$
 such that $S_\alpha\setminus Y_\xi$ is finite if $\alpha\in A_\xi$, $S_\alpha\setminus Z_b$ is finite if $\alpha\in B_\xi$ and
$S_\alpha\cap (Y_\xi\cup Z_\xi)$ is finite if $\alpha\not\in A_\xi\cup B_\xi$.
Define $a_\xi\in \ell_\infty(D)$ by
\[a_\xi(n)=\begin{cases}
p&\text{if } n\in Y_\xi, \\
q&\text{if } n\in Z_\xi,\\
0&\text{otherwise}.\end{cases}\]
Note that $a_\xi$ has a (norm) limit on $S_\alpha$, for each $\xi<\omega_2$ and $\alpha<\omega_1$. 
 Thus we can define a homomorphism $\pi$ on the C*-subalgebra $B$
 generated by $(a_\xi)_{\xi<\omega_2}$ by
\[\pi(a)(\alpha)=\lim_{n\in S_\alpha}a(n).\]
As $\pi(a_\xi)=\PP_\xi$
from Definition \ref{def-family-algebra} of the  algebra $A_\F$  for each $\xi<\omega_2$ ($\PP_\xi=\PP_i$
for $\I=\omega_2$), we have that $\pi[B]=A_{\F}$.  
By Theorem \ref{ad-algebra}, $A_{\F}$ is not $\ll$-unital so neither is $B$, by Lemma \ref{aiau-quotient}.

To prove that $B$ is LF it is enough to prove that 
C*-subalgebras of $B$ generated by finitely many $a_\xi$, 
for $\xi<\omega_2$, are LF.  But each $a_\xi$ is constant
 on $Y_b$, $Z_b$ and $\N\setminus(Y_b\cup Z_b)$.  
Thus it follows that any $a\in C^*(a_{\xi_1},\cdots,a_{\xi_n})$
 is constant on the corresponding intersections of these sets.  
As there are at most $3^n$ of these, we see that
 $C^*(a_{\xi_1},\cdots,a_{\xi_n})$ is isomorphic to a
 C*-subalgebra of $D^{3^n}$.  But $D$ is a scattered C*-algebra by
 Lemma \ref{scattered} and hence so is every finite power of $D$ 
(e.g. because scattered C*-algebras are precisely those whose
 selfadjoint elements all have countable spectrum \textendash\, 
see \cite[1.4]{cantor-bendixson}).  Hence the C*-subalgebra of 
$D^{3^n}$ isomorphic to $C^*(p_{\xi_1},\cdots,p_{\xi_n})$ is LF,
 by \cite{kusuda} (where LF is called AF).
\hfill$\Box$
\vskip 13pt

To obtain a subalgebra of $\B(\ell_2)$ like in Theorem \ref{Bell2} which is additionally scattered,
we need a separated family $\mathcal X$ with additional  properties which do not follow
from the existence of a $Q$-set. To obtain the consistency of the hypothesis of the next
example we start with a model 
 of ZFC where there is a Canadian tree $T$ and $2^{\omega_1}=\omega_2$,
for example the constructible universe. Now we obtain a generic extension as in \cite{bfz}.
The forcing used is c.c.c., so it preserves cardinals, in particular $T$ remains a Canadian tree.
In the generic extension we have Martin's Axiom  for $\sigma$-linked forcings, $2^\omega=\omega_2$ and every
Boolean algebra of cardinality $2^\omega$ embeds into the Boolean algebra $\wp(\N)/Fin$
of subsets of $\N$ modulo finite sets.  One notes that $2^{\omega_1}=2^\omega$
in this model since already Martin's axiom for countable partial orders implies it (16.20 of \cite{jech}),
so it follows that there is a Boolean embedding $E: \wp(T)\rightarrow \wp(\N)/Fin$
as $T$  has cardinality $\omega_1$.

\begin{example}\label{Bell2scattered}
A $2$-subhomogeneus C*-subalgebra of
  $\B(\ell_2)$  which is scattered but  not $\ll$-unital, in particular not AF under the hypothesis
of the existence of a Canadian tree and a Boolean embedding of $\wp(\omega_1)$ into $\wp(\N)/Fin$.
\end{example}

\noindent{\it Construction.}
Let $T$ be a Canadian tree and  $E: \wp(T)\rightarrow \wp(\N)/Fin$ a Boolean embedding.
Let $A_{\F(T)}$ be as in Example \ref{ex-tree-can}.
Let $\mathcal S=\{S_t: t\in T\}\subseteq \wp(\N)$ be given by any choice of $S_t\in E(\{t\})$
for all $t\in T$.  Note that any elements
$Z_{\mathcal X}\in E(\mathcal X)$ for $\mathcal X\subseteq T$ witness the  the fact that $\mathcal S$
is a separated family. So construct $B$ as in Example \ref{Bell2}.

Let $\pi: B\rightarrow A_{\F(T)}$ be the quotient map.  
It suffices to prove that $Ker(\pi)$ is a scattered C*-algebra,
 as extensions of scattered algebras by scattered algebras are themselves scattered, 
by \cite[Theorem 1.4 (2) or (3)]{cantor-bendixson}.  
Specifically, we prove that $Ker(\pi)$ is isomorphic to a subalgebra of $c_0(D)$.  
This is scattered because $D$ is scattered (Lemma \ref{scattered})
 and $\sigma(a)=\{0\}\cup\bigcup_n\sigma(a(n))$, 
for all self-adjoint $a\in c_0(D)$, so every $a\in c_0(D)$
 has countable spectrum (which characterizes scattered C*-algebras,
 again by \cite[1.4]{cantor-bendixson}).

First note that $Ker(\pi)$ contains $B\cap c_0(D)$, as $\pi$ is defined 
by limits on infinite subsets (in fact, we could let $B$ be generated by
 $(a_b)_{b\in Br_{\omega_1}(T)}$ and $c_0(D)$, without affecting
 the definition of $\pi$, and then $c_0(D)\cap B=c_0(D)$).  
Thus $\pi$ induces a homomorphism $\pi'$ on $B/(B\cap c_0(D))$, 
which we claim is an isomorphism.  For this it suffices to show that $\pi'$ is
 norm preserving on the image in $B/(B\cap c_0(D))$ of the
 dense *-subalgebra generated by $(a_b)_{b\in Br_{\omega_1}(T)}$.  
So take $a\in C^*(a_{b_1},\cdots,a_{b_n})\subseteq B$, for some
 $b_1,\cdots, {b_n}\in Br_{\omega_1}(T)$.  As in the argument for 
LF in Example \ref{Bell2}, the element $a$ is constant on 
the sets taken from a finite partition of $\N$ obtained by intersecting the
 sets $Y_{b_i}$, $Z_{b_i}$ and $\N\setminus(Y_{b_i}\cup Z_{b_i})$
 for $1\leq i\leq n$.  On one of these sets, say $X\subseteq\N$,
 the norm of the element $a$ assumes the norm of $[a]_{c_0(D)}$.  
As $[Y_{b_i}]_{Fin}$, $[Z_{b_i}]_{Fin}$ and $[\N\setminus(Y_{b_i}\cup Z_{b_i})]_{Fin}$ 
are in $E[\wp(T)]$ for $1\leq i\leq n$ (in fact they are $E(b_i)$, $E(b_i^\#)$ 
and $E(T\setminus(b\cup b^\#)$), $[X]_{Fin}$ is in $E[\wp(T)]$ as well.
  It follows that we have $t\in T$ such that $E(\{t\})\leq[X]_{Fin}$
 because below any nonzero element of $\wp(T)$ we have an
 element of the form $\{t\}$. This means that $S_t\setminus X$ is finite. 
 Using the definition of $\pi(a)(t)$ as $\lim_{n\in S_t}a(n)$, we conclude 
that $\|\pi(a)\|\geq\|[a]_{c_0(D)}\|$.  The reverse inequality
 is immediate so $\pi'$ is norm preserving.
\hfill$\Box$
\vskip 26pt

%\newpage

\bibliographystyle{amsplain}

\end{document}